\documentclass{amsart}
\usepackage{amsmath}
\usepackage{bbm}
\usepackage{latexsym}
\usepackage{xcolor}

\usepackage{mathrsfs}
\usepackage{bbm}

\usepackage[bookmarks=false]{hyperref}
\usepackage{mathtools} 
\usepackage[mathcal]{euscript}

 \usepackage[usenames,dvipsnames]{pstricks}%
 \usepackage{epsfig} 
 \usepackage{pst-grad} 
 \usepackage{pst-plot} 

\theoremstyle{plain}
\newtheorem{theorem}{Theorem}[section]
\newtheorem*{theorem*}{Theorem}
\newtheorem{prop}[theorem]{Proposition}
\newtheorem{cor}[theorem]{cor}

\newtheorem{lemma}[theorem]{Lemma}
\newtheorem*{lemma*}{Lemma}
\theoremstyle{definition}
\newtheorem{definition}[theorem]{Definition}
\theoremstyle{remark}

\newtheorem{remark}[theorem]{Remark}

\newcommand{\bi}{\begin{itemize}}
\newcommand{\iii}{\item}
\newcommand{\ei}{\end{itemize}}
\newcommand{\bd}{\begin{description}}
\newcommand{\ed}{\end{description}}
\newcommand{\bdeff}{\begin{definition}}
\newcommand{\edeff}{\end{definition}}
\newcommand{\bc}{\begin{cor}}
\newcommand{\ec}{\end{cor}}
\newcommand{\bl}{\begin{lemma}}
\newcommand{\el}{\end{lemma}}
\newcommand{\bp}{\begin{prop}}
\newcommand{\ep}{\end{prop}}

\newcommand{\bqn}{\begin{eqnarray}}
\newcommand{\eqn}{\end{eqnarray}}

\newcommand{\la}{\langle}
\newcommand{\ra}{\rangle}

\newcommand{\g}{\gamma}

\newcommand{\eps}{\varepsilon}			
\newcommand{\lam}{\lambda}

\newcommand{\wt}[1]{\widetilde{#1}}
\newcommand{\mc}[1]{\mathcal{#1}}
\newcommand{\distr}{\mathcal{D}}				
\newcommand{\ver}{\mathcal{V}}					
\newcommand{\J}{\mathcal{J}}					
\newcommand{\R}{\mathbb{R}}						
\newcommand{\DD}{\mathcal{F}}					
\newcommand{\tanf}{\mathsf{T}}					
\newcommand{\y}{D}								

\DeclareMathOperator{\spn}{\mathrm{span}}		

\newcommand{\bbI}{{\mathbb{I}}}
\newcommand{\bbN}{{\mathbb{N}}}
\newcommand{\bbR}{{\mathbb{R}}}
\newcommand{\bbu}{{\mathbb{U}}}
\newcommand{\tr}{{\mathrm {tr}}}

\newcommand{\diver}{{\mathrm {div}}}
\newcommand{\calD}{{\mathcal{D}}}
\newcommand{\calF}{{\mathcal{F}}}

\newcommand{\calI}{{\mathcal{I}}}
\newcommand{\calL}{{\mathcal{L}}}
\newcommand{\calN}{{\mathcal{N}}}
\newcommand{\calQ}{{\mathcal{Q}}}
\newcommand{\calR}{{\mathcal{R}}}

\newcommand{\V}{{\mathcal{V}}}

\newcommand{\tx}[1]{\mathrm{#1}}

\newcommand{\vol}{\mathrm{vol}}
\newcommand{\ad}{\tx{ad}}

\usepackage{tikz}

\mathtoolsset{showonlyrefs,showmanualtags} 
\numberwithin{equation}{section}

\author{Andrei A. Agrachev}
\address{SISSA, Via Bonomea 265, Trieste, Italy \& Steklov Math. Inst., Moscow}
\email{agrachev@sissa.it}

\author{Davide Barilari}
\address{Univ. Paris Diderot, IMJ-PRG, CNRS, Sorbonne Universit\'e.
Campus des Grands Moulins, bât. Sophie-Germain, case 7012, 75205 Paris,France}
 \email{davide.barilari@imj-prg.fr}

\author{Elisa Paoli}
\address{SISSA, Via Bonomea 265, Trieste, Italy}
\email{epaoli@sissa.it}

\date{\today}
\title{Volume geodesic distortion and Ricci curvature for Hamiltonian dynamics}
\usepackage[numbers]{natbib}

\begin{document} 
	
	\begin{abstract}
		We study the variation of a smooth volume form along extremals of a variational problem with nonholonomic constraints and an action-like Lagrangian. We introduce a new invariant, called volume geodesic derivative, describing the interaction of the volume with the dynamics and we study its basic properties. We then show how this invariant, together with curvature-like invariants of the dynamics introduced in \cite{curvature}, appear in the asymptotic expansion of the volume. This generalizes the well-known expansion of the Riemannian volume in terms of Ricci curvature to a wide class of Hamiltonian flows, including all sub-Riemannian geodesic flows.
	\end{abstract}
	
	\maketitle
	
	\setcounter{tocdepth}{1}
	\tableofcontents
	
	\newcommand{\contr}{\Omega^{T}_{x_{0}}}
	\newcommand{\K}{\mathcal{K}}

\maketitle

\section{Introduction}

One of the possible ways of introducing Ricci curvature in Riemannian geometry is by computing  the variation of the Riemannian volume under the geodesic flow. 

Given a point $x$ on a Riemannian manifold $(M,g)$ and a tangent unit vector $v\in T_xM$, it is well-known that the asymptotic expansion of the Riemannian volume $\vol_{g}$
in the direction of $v$ depends on the Ricci curvature at $x$. 
More precisely, let us consider a geodesic $\g(t)=\exp_x(tv)$ starting at $x$ with initial tangent vector $v$. In normal coordinates,  the volume element is written as $\vol_{g}=\sqrt{\det g_{ij}} dx_{1}\ldots dx_{n}$ and satisfies the following expansion for $t\to 0$ 
\begin{equation}
\sqrt{\det g_{ij}(\exp_x(tv))}=1-\frac{1}{6}\mathrm{Ric}^{g}(v,v)t^2+O(t^3),
\label{eq:volume R}
\end{equation}
where $\mathrm{Ric}^{g}$ is the Ricci curvature tensor associated with $g$ (see for instance \cite[Chapter 3]{GHL} or \cite[Chapter 14]{V}). 

The left hand side of \eqref{eq:volume R} has a clear geometric interpretation. Indeed, fix an orthonormal basis $e_1,\ldots,e_n$ in $T_xM$ and let 
$$\left.\partial_i\right|_{\g(t)}:=(d_{tv}\exp_x)(e_i)=\left.\frac{\partial}{\partial s}\right|_{s=0}\exp_x(tv+s e_i), \qquad1\leq i\leq n,$$ 
be the image of $e_i$ through the differential
of the Riemannian exponential map $\exp_x:T_{x}M\to M$ at $tv$. Once we take a set of normal coordinates centered at $x$, the vector fields $\partial_{i}$ are the coordinate vector fields at $\gamma(t)$. Then the left hand side of \eqref{eq:volume R} measures the Riemannian volume of the parallelotope with edges $\partial_i$ at the point $\g(t)$, more explicitly
$$\sqrt{\det g_{ij}(\g(t))}=\mathrm{vol}_{g}\left(\bigwedge_{i=1}^n \left.\partial_i\right|_{\g(t)}\right).$$


The purpose of this paper is to study the variation of a smooth volume form along extremals of a variational problem with nonholonomic constraints and an action-like Lagrangian. To this aim, let us first consider the case of a weighted Riemannian 
manifold $(M,g,\mu)$ endowed with a smooth volume $\mu=e^{\psi}\vol_g$, where $\psi$ is a smooth function on $M$. 
Let $\exp_x(t,v)$ denote the exponential map defined at time $t$ starting from $x$, i.e., set $\exp_x(t,v):=\exp_x(tv)$. Then
\begin{equation}
\left(d_v\exp_x(t,\cdot)\right)(e_i)=\left.\frac{\partial}{\partial s}\right|_{s=0}\exp_x(t(v+s e_i))=t\left.\partial_i\right|_{\g(t)}.
\label{eq:exp2}
\end{equation}
The volume of the parallelotope $Q_t$ with edges $t\left.\partial_i\right|_{\g(t)}$ has the following expansion for $t\to 0$, \begin{equation}
\mu\left(Q_{t}\right)=t^n e^{\psi(\g(t))}\left(1-\frac{1}{6}\mathrm{Ric}^{g}(v,v)t^2+O(t^3)\right),
\label{eq:volume R3}
\end{equation}
as a direct consequence of \eqref{eq:volume R} (see also Figure \ref{vol Riem}). By writing $$\psi(\g(t))=\psi(x)+\int_{0}^{t}g( \nabla \psi(\gamma(s)),\dot \gamma(s)) \,ds$$ we reduce the previous identity to tensorial quantities as follows
\begin{equation}
\mu\left(Q_{t}\right)=c_{0}t^n e^{\int_{0}^{t}\rho(\dot\g(s))ds}\left(1-\frac{1}{6}\mathrm{Ric}^{g}(v,v)t^2+O(t^3)\right),
\label{eq:volumeR5}
\end{equation}
where we defined $\rho(w)=g(\nabla\psi(x),w)$ for every $w\in T_{x}M$ and $c_{0}=e^{\psi(x_{0})}$.

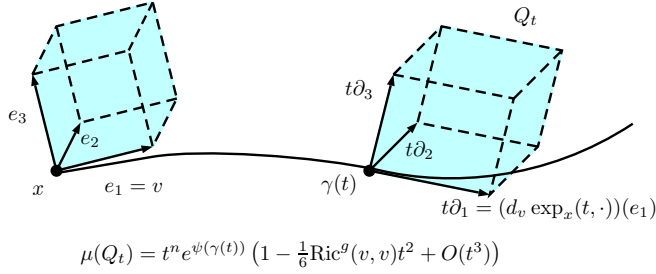
\begin{figure}[t]%
\begin{center}
\scalebox{0.8} 
{
\begin{pspicture}(0,-2.8)(11.59,2.8)
\definecolor{colour0}{rgb}{0.627451,1.0,1.0}
\pscustom[linecolor=black, linewidth=0.04]
{
\newpath
\moveto(4.4,-3.6)
}
\pscustom[linecolor=black, linewidth=0.04]
{
\newpath
\moveto(8.4,-2.8)
}
\pscustom[linecolor=black, linewidth=0.04]
{
\newpath
\moveto(9.6,-3.6)
}
\pscustom[linecolor=black, linewidth=0.04]
{
\newpath
\moveto(12.8,-4.8)
}
\pspolygon[linecolor=colour0, linewidth=0.04, strokeopacity=0.6392157, fillstyle=solid,fillcolor=colour0, opacity=0.6392157](8.0,1.2)(7.2,0.4)(6.8,-1.2)(8.8,-1.6)(9.6,-0.8)(10.0,0.8)
\pspolygon[linecolor=colour0, linewidth=0.04, strokeopacity=0.6392157, fillstyle=solid,fillcolor=colour0, opacity=0.6392157](1.6,1.2)(1.2,0.4)(1.6,-1.2)(3.2,-0.8)(3.6,0.0)(3.2,1.6)(3.2,1.6)
\rput{0.69267696}(-0.012355353,-0.07737066){\pscustom[linecolor=black, linewidth=0.04]
{
\newpath
\moveto(1.602345,-1.181362)
\lineto(3.195185,-0.9306427)
\curveto(3.9916053,-0.8052832)(5.896999,-0.94483304)(7.005973,-1.2097428)
\curveto(8.114946,-1.4746524)(9.714139,-1.423832)(11.184796,-0.47664222)
}}
\psdots[linecolor=black, dotsize=0.2](1.6,-1.2)
\psline[linecolor=black, linewidth=0.04, arrowsize=0.05291666666666668cm 2.0,arrowlength=1.4,arrowinset=0.0]{->}(1.6,-1.2)(3.2,-0.8)
\psline[linecolor=black, linewidth=0.04, arrowsize=0.05291666666666668cm 2.0,arrowlength=1.4,arrowinset=0.0]{->}(1.6,-1.2)(2.0,-0.4)
\psline[linecolor=black, linewidth=0.04, arrowsize=0.05291666666666668cm 2.0,arrowlength=1.4,arrowinset=0.0]{->}(1.6,-1.2)(1.2,0.4)
\rput{3.3891926}(-0.09249012,-0.0737401){\rput[bl](1.2,-1.6){$x$}}
\rput[bl](2.4,-1.6){$e_1=v$}
\rput[bl](2.0,-0.8){$e_2$}
\rput[bl](0.8,-0.4){$e_3$}
\rput[bl](9.2,1.2){$Q_t$}
\rput[bl](8.0,-2.0){$t\partial_1=(d_v\exp_x(t,\cdot))(e_1)$}
\rput[bl](6.4,0.0){$t\partial_3$}
\rput[bl](7.4,-1.0){$t\partial_2$}
\rput[bl](6.0,-1.6){$\gamma(t)$}
\rput[bl](2.0,-2.8){$\mu(Q_t)=t^n e^{\psi(\gamma(t))} \left(1-\frac{1}{6}\mathrm{Ric}^{g}(v,v)t^2+O(t^3)\right)$}
\psline[linecolor=black, linewidth=0.04, linestyle=dashed, dash=0.17638889cm 0.10583334cm](1.2,0.4)(2.8,0.8)(3.2,-0.8)(3.6,0.0)(3.2,1.6)(1.6,1.2)(1.2,0.4)(1.2,0.4)
\psline[linecolor=black, linewidth=0.04, linestyle=dashed, dash=0.17638889cm 0.10583334cm](1.6,1.2)(2.0,-0.4)(3.6,0.0)
\psline[linecolor=black, linewidth=0.04, linestyle=dashed, dash=0.17638889cm 0.10583334cm](3.2,1.6)(2.8,0.8)
\psdots[linecolor=black, dotsize=0.2](6.8,-1.2)
\psline[linecolor=black, linewidth=0.04, arrowsize=0.05291666666666668cm 2.0,arrowlength=1.4,arrowinset=0.0]{->}(6.8,-1.2)(8.8,-1.6)
\psline[linecolor=black, linewidth=0.04, arrowsize=0.05291666666666668cm 2.0,arrowlength=1.4,arrowinset=0.0]{->}(6.8,-1.2)(7.6,-0.4)
\psline[linecolor=black, linewidth=0.04, arrowsize=0.05291666666666668cm 2.0,arrowlength=1.4,arrowinset=0.0]{->}(6.8,-1.2)(7.2,0.4)
\psline[linecolor=black, linewidth=0.04, linestyle=dashed, dash=0.17638889cm 0.10583334cm](7.2,0.4)(8.0,1.2)(10.0,0.8)(9.6,-0.8)(8.8,-1.6)(9.2,0.0)(7.2,0.4)
\psline[linecolor=black, linewidth=0.04, linestyle=dashed, dash=0.17638889cm 0.10583334cm](8.0,1.2)(7.6,-0.4)(9.6,-0.8)(9.6,-0.8)
\psline[linecolor=black, linewidth=0.04, linestyle=dashed, dash=0.17638889cm 0.10583334cm](9.2,0.0)(10.0,0.8)(10.0,0.8)
\end{pspicture}
}
\caption{Volume distortion on a weighted Riemannian manifold with volume $\mu=e^{\psi}\vol_g$}%
\label{vol Riem}%
\end{center}
\end{figure}

To understand the general case, it is convenient to reinterpret the last variation of volume from an Hamiltonian viewpoint. Indeed the Riemannian exponential map on $M$ can be written in terms of the Hamiltonian flow associated with the smooth function $H:T^{*}M\to \R$ given in coordinates by
$$H(p,x)=\frac12\sum_{i,j=1}^{n}g^{ij}(x)p_{i}p_{j},\qquad (p,x)\in T^{*}M,$$
where $g^{ij}$ is the inverse matrix of the metric $g$. More geometrically, the Riemannian metric $g$ induces a canonical linear isomorphism $\mathfrak{i}:T_{x}M\to T_{x}^{*}M$ between each tangent space $T_xM$ and its dual $T_x^*M$. The function $H$ is then (one half of the square of) the cometric, i.e., the metric $g$ read as a function on covectors. If $\lambda=\mathfrak{i}(v)$ denotes the element in $T_x^*M$ corresponding to $v\in T_{x}M$ under the above isomorphism, the exponential map satisfies
\begin{equation}
\gamma(t)=\exp_x(t,v)=\pi( e^{t\vec{H}}(\lambda)),
\label{eq:exp}
\end{equation}
where $\pi:T^{*}M\to M$ is the canonical projection and $\vec{H}$ is the Hamiltonian vector field on $T^{*}M$ associated with $H$, whose coordinate expression is
\begin{equation}
\vec{H}=\sum_{j=1}^{n}\frac{\partial H}{\partial p_{j}}\frac{\partial }{\partial x_{j}}-\frac{\partial H}{\partial x_{j}}\frac{\partial }{\partial p_{j}}.
\end{equation}
Denote now $E_i:=\mathfrak{i}(e_{i})$ the frame of cotangent vectors in $T_x^*M$ associated with the orthonormal frame $\{e_{i}\}_{i=1}^{n}$ of $T_{x}M$. Then, combining \eqref{eq:exp2} and \eqref{eq:exp}, we have $t\partial_i= (\pi\circ e^{t\vec{H}})_*E_i$ and the left hand side of \eqref{eq:volume R3} can be written as
\begin{equation}
\begin{split}
\mathrm{\mu}(Q_t)&=\langle\mu_{\g(t)},(t\partial_1,\ldots,t\partial_n)\rangle\\
&=\langle\mu_{\pi( e^{t\vec{H}}(\lambda))},\left((\pi\circ e^{t\vec{H}})_*E_1,\ldots,(\pi\circ e^{t\vec{H}})_*E_n\right)\rangle\\
&=\langle (\pi\circ e^{t\vec{H}})^*\mu|_{\lambda},(E_1,\ldots,E_n)\rangle.
\label{eq:volume cotangent}
\end{split}
\end{equation}
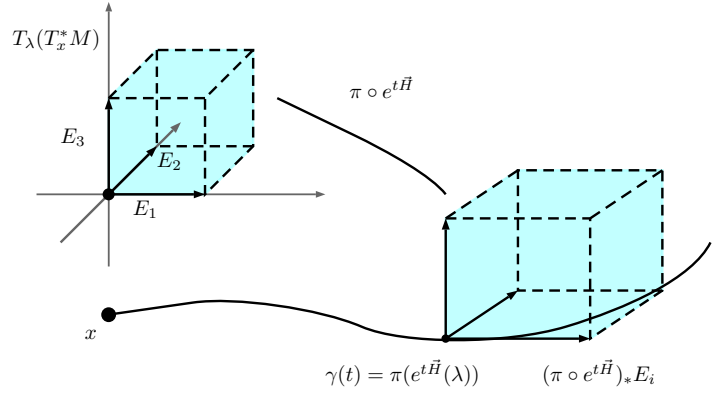
\begin{figure}[t]%
\scalebox{0.8}
{
\begin{pspicture}(0,-3.8)(12.017892,3.8)
\definecolor{colour0}{rgb}{0.627451,1.0,1.0}
\definecolor{colour1}{rgb}{0.4,0.4,0.4}
\pspolygon[linecolor=colour0, linewidth=0.04, strokeopacity=0.6392157, fillstyle=solid,fillcolor=colour0, opacity=0.6392157](7.6,-1.0)(8.8,-0.2)(11.2,-0.2)(11.2,-2.2)(10.0,-3.0)(7.6,-3.0)
\pspolygon[linecolor=colour0, linewidth=0.04, strokeopacity=0.6392157, fillstyle=solid,fillcolor=colour0, opacity=0.6392157](2.0,1.0)(2.0,-0.6)(3.6,-0.6)(4.4,0.2)(4.4,1.8)(2.8,1.8)(2.8,1.8)
\psline[linecolor=colour1, linewidth=0.03, arrowsize=0.05291666666666668cm 2.0,arrowlength=1.4,arrowinset=0.0]{->}(0.8,-0.6)(5.6,-0.6)
\psline[linecolor=colour1, linewidth=0.03, arrowsize=0.05291666666666668cm 2.0,arrowlength=1.4,arrowinset=0.0]{->}(2.0,-1.8)(2.0,2.6)
\psline[linecolor=colour1, linewidth=0.04, arrowsize=0.05291666666666668cm 2.0,arrowlength=1.4,arrowinset=0.0]{->}(1.2,-1.4)(3.2,0.6)
\psline[linecolor=black, linewidth=0.04, arrowsize=0.05291666666666668cm 2.0,arrowlength=1.4,arrowinset=0.0]{->}(2.0,-0.6)(3.6,-0.6)
\psline[linecolor=black, linewidth=0.04, arrowsize=0.05291666666666668cm 2.0,arrowlength=1.4,arrowinset=0.0]{->}(2.0,-0.6)(2.8,0.2)
\psline[linecolor=black, linewidth=0.04, arrowsize=0.05291666666666668cm 2.0,arrowlength=1.4,arrowinset=0.0]{->}(2.0,-0.6)(2.0,1.0)
\psline[linecolor=black, linewidth=0.04, linestyle=dashed, dash=0.17638889cm 0.10583334cm](2.0,1.0)(2.8,1.8)(4.4,1.8)(4.4,0.2)(3.6,-0.6)(3.6,1.0)(4.4,1.8)(4.4,1.8)
\psline[linecolor=black, linewidth=0.04, linestyle=dashed, dash=0.17638889cm 0.10583334cm](3.6,1.0)(2.0,1.0)(2.0,1.0)
\psline[linecolor=black, linewidth=0.04, linestyle=dashed, dash=0.17638889cm 0.10583334cm](2.8,1.8)(2.8,0.2)(4.4,0.2)(4.4,0.2)
\psdots[linecolor=black, dotsize=0.2](2.0,-0.6)

\rput[bl](0.4,1.8){$T_{\lambda}(T^*_{x}M)$}
\rput[bl](2.4,-1.0){$E_1$}
\rput[bl](2.8,-0.2){$E_2$}
\rput[bl](1.2,0.2){$E_3$}
\rput[bl](6.0,1.0){$\pi\circ e^{t\vec{H}}$}
\rput[bl](9.2,-3.8){$(\pi\circ e^{t\vec{H}})_*E_i$}
\psdots[linecolor=black, dotsize=0.14](7.6,-3.0)
\rput[bl](5.6,-3.8){$\gamma(t)=\pi (e^{t\vec{H}}(\lambda))$}
\rput[bl](1.6,-3.0){$x$}
\psline[linecolor=black, linewidth=0.04, linestyle=dashed, dash=0.17638889cm 0.10583334cm](7.6,-1.0)(10.0,-1.0)(10.0,-3.0)(11.2,-2.2)(11.2,-0.2)(8.8,-0.2)(8.8,-2.2)(11.2,-2.2)
\psline[linecolor=black, linewidth=0.04, arrowsize=0.05291666666666668cm 2.0,arrowlength=1.4,arrowinset=0.0]{->}(7.6,-3.0)(8.8,-2.2)
\psline[linecolor=black, linewidth=0.04, arrowsize=0.05291666666666668cm 2.0,arrowlength=1.4,arrowinset=0.0]{->}(7.6,-3.0)(7.6,-1.0)
\pscustom[linecolor=black, linewidth=0.04]
{
\newpath
\moveto(2.0,-2.6)
\lineto(3.4,-2.4)
\curveto(4.1,-2.3)(5.5,-2.5)(6.2,-2.8)
\curveto(6.9,-3.1)(8.6,-3.1)(9.6,-2.8)
\curveto(10.6,-2.5)(11.7,-2.0)(12.0,-1.4)
}
\psdots[linecolor=black, dotsize=0.24641976](2.0,-2.6)
\psline[linecolor=black, linewidth=0.04, linestyle=dashed, dash=0.17638889cm 0.10583334cm](8.8,-0.2)(7.6,-1.0)
\psline[linecolor=black, linewidth=0.04, linestyle=dashed, dash=0.17638889cm 0.10583334cm](11.2,-0.2)(10.0,-1.0)
\pscustom[linecolor=black, linewidth=0.04]
{
\newpath
\moveto(4.8,1.0)
\lineto(6.0444446,0.4)
\curveto(6.666667,0.1)(7.366667,-0.3)(7.6,-0.6)
}
\psline[linecolor=black, linewidth=0.04, arrowsize=0.05291666666666667cm 2.0,arrowlength=1.4,arrowinset=0.0]{->}(7.6,-3.0)(10.0,-3.0)
\end{pspicture}
}
\caption{Volume distortion under the Hamiltonian flow}%
\label{subRiem}%
\end{figure}

We stress that in the last formula $E_{i}$, which is an element of $T^{*}_{x}M$ is treated as a tangent vector to the fiber, i.e., an element of $T_{\lam}(T^{*}_{x}M)$ (see Figure \ref{subRiem}). Indeed the pull-back $(\pi \circ e^{t\vec{H}})^*\mu$ defines an $n$-form on $T^*M$, that has dimension $2n$, and the quantity that we compute is the restriction of this form to the $n$-dimensional vertical space $\V_{\lam}\simeq T_{\lam}(T^*_{x}M)$. 

To write a coordinate-independent formula, we compare this volume with the volume $\mu^{*}_\lam$ defined naturally on the fiber $\V_{\lam}$ by the restriction of $\mu$ at $x$. 
Recall that given a smooth volume form $\mu$ on $M$ its value $\mu_{x}$ at a point is a nonzero element of the one-dimensional vector space $\Lambda^{n}(T_{x}M)$. We can associate with it the unique element $\mu_{x}^{*}$ in $\Lambda^{n}(T_{x}M)^{*}=\Lambda^{n}(T^{*}_{x}M)$ satisfying $\mu_{x}^{*}(\mu_{x})=1$. This defines a volume form on the fiber $T^{*}_{x}M$. By the canonical identification $T^{*}_{x}M\simeq T_{\lam}(T^{*}_{x}M)$ of a vector space with its tangent space to a point, this induces a volume form $\mu^{*}_{\lam}$ on $\V_{\lambda}$.

With this interpretation, the Riemannian asymptotics \eqref{eq:volumeR5} computes the asymptotics in $t$ of $(\pi\circ e^{t\vec{H}})^*\mu$ restricted to the fiber $\V_{\lam}$, with respect to the volume $\mu^{*}_\lambda$, i.e.,
\begin{equation}\label{eq:Riem33}
\left.(\pi\circ e^{t\vec{H}})^*\mu\right|_{\V_{\lambda}}=t^n\; e^{\int_0^t\rho(\dot \g(s))ds}\left(1-\frac{1}{6}\mathrm{Ric}^{g}(v,v)t^2+O(t^3)\right)\mu^{*}_\lambda.
\end{equation}
The constant $c_{0}$ appearing in \eqref{eq:volumeR5} is reabsorbed in the volume $\mu_\lambda^{*}$.




\begin{remark} As it follows from \eqref{eq:Riem33}, 
the quantity $\rho$  can be equivalently characterized as follows 
\begin{equation}\label{eq:rhoriem}
\rho(v)\mu^{*}_\lambda=\left.\frac{d}{dt}\right|_{t=0}\log\left(t^{-n}\left.(\pi\circ e^{t\vec{H}})^*\mu\right|_{\V_{\lam}}\right).
\end{equation}
where $\lambda=\mathfrak{i}(v)$. The last formula inspires indeed the definition of $\rho$ in the general case.
\end{remark} 

Identity \eqref{eq:Riem33} can be generalized to every Hamiltonian that is quadratic and convex on fibers, giving a suitable meaning to the terms in the right hand side.
More precisely, we consider Hamiltonians $H:T^{*}M\to \R$ of the following form 
\begin{equation}\label{eq:hamintro}
H(p,x)=\frac{1}{2}\sum_{i=1}^{k} \la p, X_{i}(x)\ra^{2}+\la p,X_0(x)\ra+\frac{1}{2}Q(x),\qquad (p,x)\in T^{*}M.
\end{equation}
where $X_{0},X_{1},\ldots,X_{k}$ are smooth vector fields on $M$ and $Q$ is a smooth function. We assume that 
\begin{itemize}
\iii[(H0)]
$X_1,\ldots,X_k$ are everywhere linearly independent,
\iii[(H1)] $\mathrm{Lie}\{(\ad \,X_{0})^{j}X_{i}\,|\, i=1,\ldots,k, j\geq 0\}\big|_{x}=T_{x}M$ for every $x\in M$.
\end{itemize}
where $(\ad \,Y)X=[Y,X]$ and $\mathrm{Lie}\, \mathcal{F}$ denotes the smallest Lie algebra containing a family of vector fields $\mathcal{F}$.
 The Hamitonian \eqref{eq:hamintro} is naturally associated with the optimal control problem where the dynamics  is affine in the control 
\begin{equation}\label{eq:controlsyst0}
\dot{x}(t)=X_0(x(t))+\sum_{i=1}^k u_i(t) X_i(x(t)), \qquad x\in M, 
\end{equation}
and one wants to minimize a quadratic cost with potential (here $x_{u}$ denotes the solution of \eqref{eq:controlsyst0} associated with $u\in L^{\infty}([0,T],\R^{k})$)
\begin{equation}\label{eq:J}
J_T(u):=\frac{1}{2}\int_0^T \|u(s)\|^2-Q(x_{u}(s))ds.
\end{equation}
We stress that when $X_{0}=0$, $Q=0$ and $k=n$, the optimal control problem described above is the geodesic problem associated with the Riemannian metric defined by the orthonormal frame $X_{1},\ldots,X_{n}$ on $M$ and $H$ is the corresponding Hamiltonian. The case $X_{0}=0$, $Q=0$ and $k<n$ includes the geodesic problem in sub-Riemannian geometry.

Consider the projections on $M$ of integral curves of the Hamiltonian vector field $\vec H$ in $T^*M$. Under our assumptions, short pieces of these curves are minimizers for the optimal control problem (i.e., geodesics in the case of Riemannian or sub-Riemannian geometry). However, in general, not all minimizers can be obtained in this way. This is due to the possible presence of the so-called abnormal minimizers \cite{montgomeryabnormal}. The projected trajectories, as solutions of an Hamiltonian system in $T^{*}M$, are smooth and parametrized by the initial covector in the cotangent bundle. 

If the initial covector $\lambda$ corresponds to an \emph{ample} and \emph{equiregular} trajectory (cf.\
 Section \ref{s:gfyd} for precise definitions) then the exponential map $\pi\circ e^{t\vec H}$ is a local diffeomorphism at $\lambda$ and it is possible to compute the variation of a smooth volume $\mu$ under the exponential map, as in the Riemannian case.

Let us stress that in the Riemannian case all $\lambda\in T^{*}M$ satisfy these assumptions, while in the sub-Riemannian case one can prove that there exists a non-empty open and dense set of covectors $\mathcal{A}\subset T^{*}M$ such that the corresponding geodesic is ample and equiregular (see Proposition \ref{p:aequi}). If the initial point $x$ is fixed, then there exists a non-empty Zariski open set of covectors in $T^{*}_{x}M$ such that the corresponding geodesic is ample, but the existence of equiregular geodesics is not guaranteed. On the other hand, on any ample geodesic, there exists an open dense set of $t$ such that the germ of the geodesic at $\gamma(t)$ is equiregular (cf.\ Section \ref{s:sr}). 
 

\medskip 
The main result of this paper is the generalization of the asymptotics \eqref{eq:Riem33} to any flow arising from an  Hamiltonian that is quadratic and convex on fibers, along any trajectory satisfying our regularity assumptions. In particular we give a geometric characterization of the terms appearing in the asymptotic expansion of a smooth volume $\mu$ under the Hamiltonian flow $\pi\circ e^{t\vec H}$ and we interpret every coefficient as the generalization of the corresponding Riemannian element. 


Fix $x\in M$ and $\lambda\in T_x^*M$. Let $\gamma(t)=\pi(e^{t\vec{H}}(\lambda))$ be the associated geodesic on $M$ and assume that it is ample and equiregular. With such a geodesic it is possible to associate an integer $\calN(\lambda)$  which is defined through the structure of the Lie brackets of the controlled vector fields $X_{1},\ldots,X_{k}$ along $\g$ (cf.\ Definition \ref{def:geodesic dimension}). This is in an invariant that depends only on the Lie algebraic structure of the drift field $X_{0}$ and the distribution $\distr=\text{span}\{X_{1},\ldots,X_{k}\}$ \emph{along the trajectory} and not on the particular frame (that induces the metric) on $\distr$. The notation stresses that this integer can a priori depend on the trajectory.

The results obtained in \cite[Section 6.5]{curvature} imply that there exists $C_\lambda>0$ such that, for $t\to 0$
\begin{equation}\label{eq:SRiem33}
\left.(\pi\circ e^{t\vec{H}})^*\mu\right|_{\V_{\lam}}=t^{\calN(\lambda)}\left(C_\lambda+O(t)\right)\mu^{*}_\lam.
\end{equation} 
Once the order of the asymptotics is determined, one can introduce the \emph{volume geodesic derivative} 
as follows 
\bqn
\rho(\lam)\mu^{*}_{\lam}:=\frac{d}{dt}\bigg|_{t=0} \log \left(t^{-\mathcal{N}(\lambda)}\left.(\pi\circ e^{t\vec H})^{*}\mu\right|_{\V_{\lam}}\right),
\eqn
This is the natural generalization of the quantity introduced in \eqref{eq:rhoriem}.  In Section \ref{s:formula} we show an explicit formula to compute $\rho$, in terms of the \emph{symbol of the structure along the geodesic} (cf.\ Definition~\ref{def:symbol}) and we compute it explicitly for contact manifolds endowed with Popp's volume. 

The Riemannian Ricci tensor appearing in \eqref{eq:Riem33} is  replaced by the trace of a curvature operator in the direction of $\lambda$. This curvature operator $\mc{R}_\lambda:\distr_{x}\to \distr_{x}$, is a generalization of the sectional curvature and is defined in \cite{curvature} for the wide class of geometric structures arising from affine control systems. In the Riemannian case $\mc{R}_\lambda(w)=R^{g}(w,v)v$, where $R^{g}$ is the Riemann tensor associated with $g$, $\lambda=\mathfrak{i}(v)$ and $w\in \distr_{x}=T_{x}M$. Notice that in this case $\mathrm{tr} \, \mc{R}_\lambda=\text{Ricci}^{g}(v,v)$.
 
All the geometric invariants are rational functions in the initial covector $\lambda$. The precise statement of our theorem reads as follows.
\begin{theorem}\label{th:main} Let $\mu$ be a smooth volume form on $M$ and $\g(t)=\pi(\lam(t))=\pi(e^{t\vec H}(\lam))$ be an ample and equiregular trajectory, with $\lambda\in T^*_xM$.  Then we have the following asymptotic expansion
\begin{equation} \label{eq:volume2}
\left.(\pi\circ e^{t\vec H})^{*}\mu\right|_{\V_{\lam}}=C_{\lam}t^{\mc{N}(\lambda)}\;e^{\int_0^t \rho(\lambda(s))ds}\;\left(1-t^2\frac{\mathrm{tr} \mc{R}_\lambda}{6}+o(t^2) \right)\mu^*_{\lam}.
\end{equation}
where $\mu_\lam^{*}$ is the canonical volume induced by $\mu$ on $\V_{\lam}\simeq T_{\lam}(T^{*}_{x}M)$.
\end{theorem}
Again, in the asymptotics \eqref{eq:volume2}, the choice of the volume form $\mu$ affects only the volume geodesic derivative $\rho$. Notice also that the constant $C_{\lambda}$ and the main order $t^{\calN(\lam)}$ depend only on the Young diagram associated with the curve $\g$, while the term $\mc{R}_\lambda$ (and actually the whole asymptotic expansion contained in the parentheses) depends only on the curvature like-invariants of the cost of the optimal control problem \eqref{eq:controlsyst0}-\eqref{eq:J}, i.e., on the Hamiltonian \eqref{eq:hamintro}. 

We stress that in the Riemannian case every trajectory is ample and equiregular and the constant $C_{\lam}$ is always equal to one and $\calN(\lam)=n=\dim M$ (cf.\ formula \eqref{eq:Riem33}).

In other words, the asymptotics \eqref{eq:volume2} ``isolates'' the contribution given by the volume form with respect to the contribution given by the dynamics/geometry.

\begin{remark}
 Another geometric interpretation of the variation of the volume is given in Figure \ref{figurabellissima}. Let $\Omega\subset T^*_xM$ be an infinitesimal neighborhood of $\lambda$ and let $\Omega_{x,t}:=\pi\circ e^{t\vec{H}}(\Omega)$ be its image on $M$ with respect to the Hamiltonian flow. For every $t$ the set $\Omega_{x,t}\subset M$ is a neighborhood of $\g(t)$. By construction
$$\mu(\Omega_{x,t})=\int_{\Omega}(\pi\circ e^{t\vec{H}})^*\mu,$$
and \eqref{eq:volume2} represents exactly the variation of the volume element along $\g$.
\end{remark}

\begin{figure}[h]
\begin{center}
\scalebox{0.8} 
{
\begin{pspicture}(0,-4.763445)(15.632034,4.763445)
\definecolor{colour1}{rgb}{0.627451,1.0,1.0}
\definecolor{colour2}{rgb}{0.6,0.6,0.6}
\psbezier[linecolor=black, linewidth=0.02, linestyle=dashed, dash=0.17638889cm 0.10583334cm, fillstyle=solid,fillcolor=colour1, opacity=0.29803923](7.617892,-0.20155518)(7.617892,-0.55155516)(8.455439,-0.8613052)(8.517892,-0.55155516)(8.580345,-0.24180518)(8.292892,-0.4348885)(8.509455,-0.087805174)(8.726017,0.25927815)(7.9864087,0.21020953)(7.8849225,0.14844483)(7.783436,0.08668012)(7.617892,0.14844483)(7.617892,-0.20155518)
\psbezier[linecolor=black, linewidth=0.02, linestyle=dashed, dash=0.17638889cm 0.10583334cm, fillstyle=solid,fillcolor=colour1, opacity=0.29803923](5.457892,-0.4391385)(5.4812255,-0.69385517)(5.889274,-0.73394686)(5.9359407,-0.49155518)(5.9826074,-0.24916351)(5.947892,-0.57882184)(6.070392,-0.33643016)(6.192892,-0.09403851)(5.6795588,-0.10322184)(5.6037254,-0.15155518)(5.527892,-0.19988851)(5.434559,-0.18442184)(5.457892,-0.4391385)
\psbezier[linecolor=black, linewidth=0.02, linestyle=dashed, dash=0.17638889cm 0.10583334cm, fillstyle=solid,fillcolor=colour1, opacity=0.29803923](9.577892,-0.23855518)(9.612654,-0.6167552)(10.377416,-0.9339552)(10.446939,-0.5740552)(10.516463,-0.21415518)(10.307892,-0.4459552)(10.490392,-0.086055174)(10.672892,0.27384484)(10.017892,0.4484448)(9.795154,0.18844482)(9.572415,-0.071555175)(9.54313,0.13964483)(9.577892,-0.23855518)
\rput{-9.303995}(-0.25771296,0.4131096){\psellipse[linecolor=black, linewidth=0.02, fillstyle=solid,fillcolor=colour1, opacity=0.29803923, dimen=outer](2.4095588,1.7901115)(0.34166667,0.85833335)}
\psbezier[linecolor=black, linewidth=0.04](0.017891997,-2.851555)(2.8874571,-2.0515552)(9.774414,-2.0515552)(13.217892,-2.851555)
\psbezier[linecolor=black, linewidth=0.04](13.217892,-2.851555)(14.017892,-0.85155517)(15.217892,0.34844482)(15.617892,0.7484448)
\psdots[linecolor=black, dotsize=0.16](2.417892,-0.85155517)
\psline[linecolor=black, linewidth=0.03](1.217892,2.0484447)(1.217892,-2.0515552)
\psline[linecolor=black, linewidth=0.028](3.217892,4.748445)(3.217892,0.5284448)
\psline[linecolor=black, linewidth=0.03](1.217892,2.0484447)(3.217892,4.748445)(3.217892,4.748445)
\psline[linecolor=black, linewidth=0.04](1.217892,-4.751555)(2.917892,-2.4515553)
\psbezier[linecolor=black, linewidth=0.04, linestyle=dashed, dash=0.17638889cm 0.10583334cm](1.217892,-0.85155517)(1.817892,0.108444825)(2.217892,0.62844485)(2.517892,0.94844484)
\psbezier[linecolor=black, linewidth=0.04](0.017891997,-2.851555)(0.517892,-1.8515552)(1.217892,-0.85155517)(1.217892,-0.85155517)
\psbezier[linecolor=black, linewidth=0.04](3.217892,1.1484448)(7.557892,1.5484449)(13.757892,1.1484448)(15.617892,0.7484448)
\psbezier[linecolor=black, linewidth=0.04, linestyle=dashed, dash=0.17638889cm 0.10583334cm](2.517892,0.94844484)(2.817892,1.0484449)(3.0845587,1.1484448)(3.217892,1.1484448)
\psbezier[linecolor=black, linewidth=0.02, arrowsize=0.05291666666666668cm 2.0,arrowlength=1.4,arrowinset=0.0]{->}(2.417892,1.7484448)(5.3282523,1.7484448)(8.74627,0.8140698)(9.417892,0.4484448)
\rput[bl](2.217892,1.4484448){$\lambda$}
\rput[bl](2.017892,-1.2515552){$x$}
\rput[bl](9.797892,-0.13155517){$\gamma(t)$}
\rput[bl](12.517892,-2.4515553){$M$}
\rput[bl](2.277892,3.0684447){$T_{x}^*M$}
\psdots[linecolor=black, dotsize=0.16](10.117892,-0.25155517)
\psbezier[linecolor=black, linewidth=0.04, linestyle=dashed, dash=0.17638889cm 0.10583334cm](2.417892,-0.85155517)(4.617892,-0.45155516)(9.617892,0.14844483)(12.017892,-0.45155516)
\psdots[linecolor=black, dotsize=0.1](8.034492,-0.22815518)
\psdots[linecolor=black, dotsize=0.1](5.7345586,-0.3915552)
\psbezier[linecolor=colour2, linewidth=0.02, linestyle=dashed, dash=0.17638889cm 0.10583334cm](2.417892,-0.8265552)(3.817892,-0.25155517)(8.017892,0.24844483)(10.217892,0.34844482)
\psbezier[linecolor=colour2, linewidth=0.02, linestyle=dashed, dash=0.17638889cm 0.10583334cm](2.355392,-0.8265552)(5.417892,-0.6515552)(9.192892,-0.61405516)(10.117892,-0.7515552)
\rput[bl](8.815392,0.7409448){$\pi\circ e^{t\vec{H}}$}
\psline[linecolor=black, linewidth=0.03, linestyle=dashed, dash=0.17638889cm 0.10583334cm](2.617892,-2.851555)(3.217892,-2.026555)(3.217892,-1.9322695)(3.217892,0.40719482)(3.217892,0.4484448)
\rput[bl](12.137892,-0.81155515){$\gamma$}
\psdots[linecolor=black, dotsize=0.1](2.417892,1.7484448)
\psline[linecolor=black, linewidth=0.03](1.217892,-2.651555)(1.217892,-4.751555)
\psline[linecolor=black, linewidth=0.03, linestyle=dashed, dash=0.17638889cm 0.10583334cm](1.217892,-2.5515552)(1.217892,-2.0515552)
\rput[bl](1.817892,2.0484447){$\Omega$}
\rput[bl](9.817892,-1.1515552){$\Omega_{x,t}$}
\end{pspicture}
}
\end{center}
\caption{Infinitesimal variation of the volume along a geodesic}
\label{figurabellissima}
\end{figure}
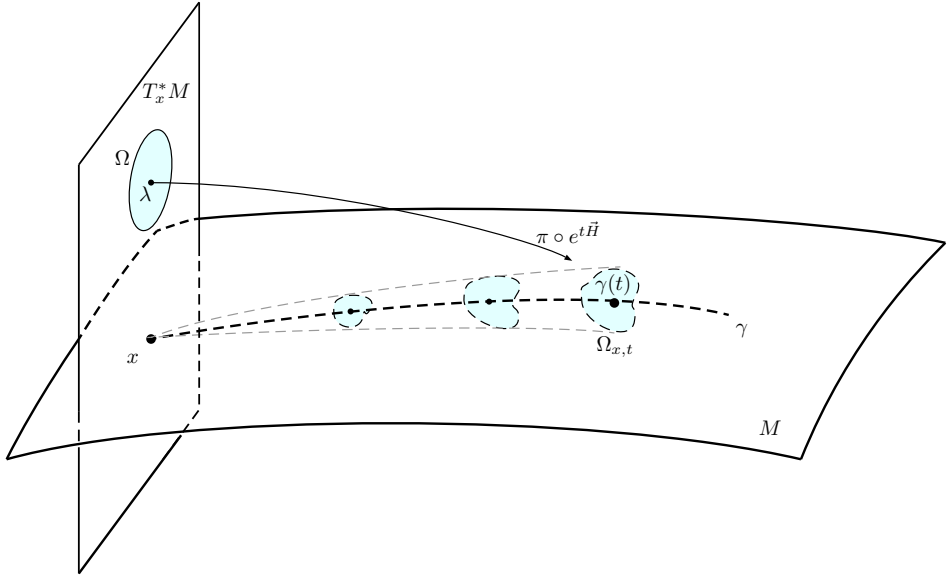
%
%

\subsection{The invariant $\rho$ in the Riemannian case}

In the Riemannian case, for $\mu=e^{\psi}\vol_{g}$, one has $\rho(v)=g(\nabla \psi,v)$ for every tangent vector $v$. 
If one writes explicitly the expansion of the exponential term in  \eqref{eq:Riem33} at order 2 one finds, for $\gamma(t)=\exp_{x}(tv)$
\begin{align}\label{eq:Riem334}
\left.(\pi\circ e^{t\vec{H}})^*\mu\right|_{\V_{\lambda}}=
t^n&\left(1+  \rho( v)t-\left(\frac{1}{6}\mathrm{Ric}^{g}(v,v)-\frac{1}{2}\rho( v)^{2}-\frac12\nabla\rho( v, v)\right)t^2+O(t^3)\right)\mu^{*}_\lambda.
\end{align}
In particular for $X,Y$ vector fields on $M$
\begin{gather}
\rho^{2}=\nabla \psi\otimes \nabla \psi,\qquad \rho(X)^2= g(\nabla \psi,X)^2, \\
\nabla\rho=\mathrm{Hess}\,\psi,\qquad \nabla\rho(X,Y)=Y(X\psi)-(\nabla_{Y}X) \psi,
\end{gather}
and summing over a local orthonormal basis $X_{1},\ldots,X_{n}$  for $g$ one can compute the traces 
\begin{gather}
\mathrm{tr}\, \rho^{2}
=
\|\nabla \psi\|^2 ,\qquad
\mathrm{tr}\, \nabla\rho=
\Delta_g \psi.
\end{gather}

\subsection{On the relation with the small time heat kernel asymptotics} The volume geodesic derivative $\rho$ introduced in this paper, together with the curvature-like invariants of the dynamics, characterize in the Riemannian case the small time heat kernel expansion on the diagonal associated with a weighted Laplacian $\Delta_{\mu}=\text{div}_{\mu}\nabla$.

Indeed let us consider a weighted Riemannian manifold $(M,g,\mu)$ with $\mu=e^{\psi}\text{vol}_{g}$, and denote by $p_{\mu}(t,x,y)$  the fundamental solution of the heat equation $\partial_{t}-\Delta_{\mu}=0$ associated with $\Delta_{\mu}$.  Recall that $\Delta_{\mu}=\Delta_{g}+g(\nabla \psi,\nabla\cdot)$, where $\Delta_{g}$ is the Laplace-Beltrami operator of $(M,g)$. One has the following small time asymptotics (see for instance \cite{bismutbook})
$$p_\mu(t,x,x)=\frac{c}{t^{n/2}}\left[1+t\left(\frac{S(x)}{12}-\frac{\|\nabla\psi(x)\|^2}{8}-\frac{\Delta_g(\psi)}{4}\right)+o(t)\right]$$
for some $c>0$. 

Hence the terms appearing in the heat kernel expansion are exactly the trace of the invariants that determine the expansion of the exponential in  \eqref{eq:Riem33} at order 2. 

A natural conjecture is that the same three coefficients describe the heat kernel small time asymptotics expansion also in the sub-Riemannian case. This conjecture is true in the 3D contact case  as it follows by the results obtained in \cite{article:Barilari}, since on a 3D contact manifold if one chooses $\mu$ equal to Popp volume one has $\rho=0$ (cf.\ Remark \ref{rem:popprho3d}). See also \cite{barilaripaoli,paoli,barboa} for some results about small time heat kernel expansion for H\"ormander operators with drift.
 
\subsection{Structure of the paper}
In Section \ref{s:setting} we describe the general setting, while in Section \ref{s:gfyd} and \ref{s:cfjf} we introduce some preliminaries needed for the proof of the main result. Section \ref{s:rho} is devoted to the definition of the volume geodesic derivative $\rho$ and the analysis of its properties, while in Section \ref{s:formula} we give a formula that permits us to compute it through an adapted frame. Section \ref{s:sr} specifies the whole construction to sub-Riemannian case. Finally, Section \ref{s:proof} contains the proof of the main result.

\section{General setting} \label{s:setting}
Let $M$ be an $n$-dimensional connected manifold and let $X_0,X_1\ldots,X_k\in \mathrm{Vec}(M)$ be smooth vector fields, with $k\leq n$. We consider the following \emph{affine control system} on $M$
\begin{equation}\label{eq:controlsyst}
\dot{x}(t)=X_0(x(t))+\sum_{i=1}^k u_i(t) X_i(x(t)), \qquad x\in M,
\end{equation}
where $u\in L^{\infty}([0,T],\R^{k})$ is a measurable and essentially bounded function called \emph{control}. In what follows we assume that 
\begin{itemize}
\iii[(H0)]
$X_1,\ldots,X_k$ are everywhere linearly independent,
\iii[(H1)] $\mathrm{Lie}\,\{(\ad \,X_{0})^{j}X_{i}\,|\, i=1,\ldots,k, j\geq 0\}\big|_{x}=T_{x}M$ for every $x\in M$.
\end{itemize}
where $(\ad \,Y)X=[Y,X]$ and $\mathrm{Lie}\,\mathcal{F}$ denotes the smallest Lie algebra containing a family of vector fields $\mathcal{F}$).

A Lipschitz curve $\g:[0,T]\to M$ is said to be \emph{admissible} for the system \eqref{eq:controlsyst} if there exists a control $u\in L^{\infty}([0,T],\R^k)$ such that $\g$ satisfies \eqref{eq:controlsyst} for a.e.\ $t\in[0,T]$. The pair $(\g,u)$ of an admissible curve $\g$ and its control $u$ is called \emph{admissible pair}.

\begin{remark}
The affine control system can be defined more generally as a pair $(\bbu, f)$, where $\bbu$ is a smooth vector bundle of rank $k$ with base $M$ and fiber $\bbu_x$, and $f:\bbu\to TM$ is a smooth affine morphism of vector bundles such that $\pi\circ f(u)=x$, for every $u\in\bbu_x$. Locally, by taking a local trivialization of $\bbu$, we can write $f(u)=X_0+\sum_{i=1}^k u_i X_i$ for $u\in\bbu$. For more details about this approach see \cite{nostrolibro,curvature}.
\end{remark}

We denote by $\calD\subset TM$ the \emph{distribution}, that is the family of subspaces  spanned by the linear part of the control problem at a point, i.e.,
$$\calD=\{\calD_x\}_{x\in M}, \quad \mbox{ where } \calD_x:=\mathrm{span}\{X_1,\ldots,X_k\}\big|_x\subset T_{x}M.$$
The distribution $\calD$ has constant rank by assumption (H0), and we endow $\calD$ with the inner product such that the fields $X_1,\ldots,X_k$ are orthonormal. 
We denote by $\Gamma(\calD)$ the set of smooth sections of $\calD$, also called \emph{horizontal} vector fields. 
Among all admissible trajectories that join two fixed points in time $T>0$, we want to minimize the \emph{quadratic cost functional}
$$J_T(u):=\frac{1}{2}\int_0^T \|u(s)\|^2-Q(x_{u}(s))ds,$$
where $Q$ is a smooth function on $M$, playing the role of a potential. Here $x_{u}$ denotes the solution of \eqref{eq:controlsyst} associated with $u$.
\begin{definition}
For $x_0,x_1\in M$ and  $T>0$, we define the \emph{value function}
\begin{equation}
S_T(x_0,x_1):=\inf\left\{J_T(u)\,|\, (\g,u) \mbox{ admissible pair}, \g(0)=x_0, \g(T)=x_1\right\}.
\label{eq:value}
\end{equation}
\end{definition}

The assumption (H1) implies, by Krener's theorem  (see \cite[Theorem 3.10]{notejak} or
 \cite[Chapter 3]{Book:Jurdjevic}), that the attainable set in time $T>0$ from a fixed point $x_{0}\in M$, that is the set  $$A_{x_{0},T}=\{x_{1}\in M: S_{T}(x_{0},x_{1})<+\infty\}$$ has non-empty interior for all $T>0$. This is a necessary assumption to the existence of ample geodesics.

%

\medskip

Important examples of affine control problems are sub-Riemannian structures. These are  triples $(M,\distr,g)$, where $M$ is a smooth manifold, $\distr$ is a smooth, completely non-integrable vector sub-bundle of $TM$ and $g$ is a smooth inner product on $\distr$. In our context, these are included in the case $X_{0}=0$ and $Q=0$.  The value function in this case coincides with (one half of the square of) the sub-Riemannian distance, i.e., the infimum of the length of absolutely continuous admissible curves joining two points. 
In this case the assumption (H1) on $\distr$ implies, by the Rashevskii-Chow theorem, that the sub-Riemannian distance is finite on $M$. Moreover the metric topology coincides with the one of $M$.  A more detailed introduction on sub-Riemannian geometry can be found in \cite{nostrolibro,montgomerybook}.

\medskip

%

For an affine optimal control system, the associated Hamiltonian is defined as follows 
\bqn\label{eq:classH}
H(p,x)=\frac{1}{2}\sum_{i=1}^{k} \la p, X_{i}(x)\ra^{2}+\la p,X_0(x)\ra+\frac{1}{2}Q(x),\qquad (p,x)\in T^{*}M.
\eqn 
Hamilton's equations are written as follows
\begin{equation}\label{eq:PMP0}
\dot x = \frac{\partial H}{\partial p},\qquad  \dot p = - \frac{\partial H}{\partial x}\qquad (p,x)\in T^{*}M,
\end{equation}
\begin{theorem}[PMP, \cite{Book:Agrachev,pontbook}]\label{t:pmp} \index{Pontryagin Maximum Principle (PMP)} 
Consider a solution $\lam(t)=(p(t),\gamma(t))$ defined on $[0,T]$ of the Hamilton equations \eqref{eq:PMP0} on $T^{*}
M$. Then short pieces of the trajectory $\gamma(t) = \pi(\lam(t))$ minimize the cost between their endpoints.
\end{theorem}

From now on, using a slight abuse of notation, we call \emph{geodesic} \index{geodesic} any projection $\gamma:[0,T]\to M$ of an integral line of the Hamiltonian vector field. In the general case, some minimizers of the cost might not satisfy this equation.  These are the so-called strictly abnormal minimizers \cite{montgomeryabnormal}, and they are related with hard open problems in control theory and in sub-Riemannian geometry. In what follows we will focus on those minimizers that come from the Hamilton equations (also called normal) and that satisfy a suitable regularity assumption. Notice that normal geodesics are smooth.

\section{Geodesic flag and symbol}\label{s:gfyd}
In this section we define the flag and the symbol of a geodesic, that are elements carrying information about the germ of the distribution and the drift along the trajectory. The symbol is the graded vector space associated with the flag and is endowed with an inner product induced by the metric on the distribution.

\subsection{The class of symbols} We start by defining the class of objects we deal with.
\bdeff A \emph{symbol} $S$ is a pair $(V,L)$ where 
\bi 
\iii[(i)] $V$ is a graded vector space $V=\oplus_{i=1}^{m}V_{i}$, endowed with an inner product $\la\cdot,\cdot\ra$ on its first layer $V_{1}$, 
\iii[(ii)] $L=\{L_{i}\}_{i=1}^{m}$ is a family of surjective linear maps $L_{i}:V_{1}\to V_{i}$.
\ei
\edeff

\begin{remark}
Through the surjective linear maps $L_{i}:V_{1}\to V_{i}$, the inner product on $V_{1}$ naturally induces a norm of $v\in  V_{i}$ as follows
$$\|v\|_{V_{i}}:=\min \left\{\| u\|_{V_{1}} \mbox{ s.t. } L_{i}(u)=v  \right\}.$$\
It is easy to check that, since $\|\cdot\|_{V_{1}}$ is induced by an inner product, then $\|\cdot\|_{V_{i}}$ is induced by an inner product as well. Hence the family of surjective maps endows $V$ with a global inner product by defining the subspaces $\{V_{i}\}_{i=1}^{m}$ to be mutually orthogonal.
%
\end{remark}

\bdeff\label{def:equivalencesymbols} We say that the symbols $S=(V,L)$ and $S'=(V',L')$ are \emph{isomorphic} if there exists an invertible linear map  $\phi:V\to V'$ such that $\phi|_{V_{1}}:V_{1}\to V'_{1}$ is an isometry  and $L'_{i}\circ \phi=\phi\circ L_{i}$ for $i\geq1$.
\edeff

\begin{lemma}
If two symbols $S$ and $S'$ are isomorphic, then they are isometric as inner product spaces.
\end{lemma}
\begin{proof}
Let $V=\oplus_{i=1}^{m}V_{i}$ and $V'=\oplus_{i=1}^{m'}V'_{i}$ and let $\phi$ be the map given in Definition~\ref{def:equivalencesymbols}. Let $v\in V_{1}$ and $v'=\phi (v) \in V'_{1}$. By the commutation property satisfied by $\phi$ one has
$$L'_{i}(v')=L'_{i}\left(\phi (v)\right)=\phi\left(L_{i} (v)\right),$$
therefore $V'_{i}=\phi(V_{i})$ for every $i\geq 1$ and in particular $m=m'$. As a consequence the map $\phi$ descends to a family of maps between every layer of the stratification as follows: for $v_{i}\in V_{i}$ write  $v_{i}=L_{i}(v)$ for some $v\in V_{1}$ and define $\phi_i:V_{i}\to V'_{i}$ by $\phi_{i}(v_{i}):=L'_{i}(\phi(v))$. Since $\phi$ is an isometry on $V_{1}$, then the map $ \phi|_{V_{i}}: V_{i}\to V'_{i}$  is an isometry on each layer.
\end{proof}

\subsection{The symbol of a geodesic}
Let $\gamma:[0,T]\to M$ be a  geodesic and consider a smooth admissible extension of its tangent vector, namely a vector field $\tanf=X_0+X$, with $X\in\Gamma(\calD)$, such that $\tanf(\g(t))=\dot{\g}(t)$ for every $t\in[0,T]$.

\begin{definition}\label{d:flag}
The \emph{flag of the geodesic} $\gamma:[0,T]\to M$ is the one-parameter family of subspaces
\begin{equation}
\DD_{\gamma(t)}^i :=  \spn\{\calL_\tanf^j (X)|_{\gamma(t)} \mid  X \in \Gamma(\calD),\, j \leq i-1\} \subseteq T_{\gamma(t)} M, \qquad \forall\, i \geq 1,
\end{equation}
for any fixed $t\in[0,T]$, where $\calL_{\tanf}$ denotes the Lie derivative in the direction of $\tanf$.
\end{definition}
Definition~\ref{d:flag} is well posed, namely it does not depend on the choice of the admissible extension $\tanf$ (see \cite[Sec. 3.4]{curvature}). By construction, the flag is a filtration of $T_{\gamma(t)}M$, i.e., $\DD_{\gamma(t)}^i \subseteq \DD_{\gamma(t)}^{i+1}$, for all $i \geq 1$. Moreover, $\DD_{\gamma(t)}^1 = \distr_{\gamma(t)}$. 
The \emph{growth vector} of the geodesic $\gamma(t)$ is the one-parameter family of sequences of integers 
\begin{equation}
\mathcal{G}_{\gamma(t)} := \{\dim \DD_{\gamma(t)}^1,\dim \DD_{\gamma(t)}^2,\ldots\}.
\end{equation}
A geodesic $\gamma(t)$, with growth vector $\mathcal{G}_{\gamma(t)}$, is said to be
\begin{itemize}
\item[(i)] \emph{equiregular} if, for every $i \geq 1$, the integer $\dim \DD_{\gamma(t)}^i$ does not depend on $t$,
\item[(ii)] \emph{ample} if for every $t$ there exists $m=m(t) \geq 1$ such that $\dim \DD_{\gamma(t)}^{m} = \dim T_{\gamma(t)}M$.
\end{itemize}
We define the integers $d_i(t):= \dim \DD_{\gamma(t)}^i - \dim \DD_{\gamma(t)}^{i-1}$, for $i\geq 0$, with the convention $\dim\mathcal{F}_{\g(t)}^0=0$. 

Ample (resp.\ equiregular) geodesics are the microlocal counterpart of bracket-generating  (resp.\ equiregular) distributions. 
 For an equiregular and ample geodesic the integers $d_{i}$ are independent on $t$ and represent the increment of dimension of the flag of the geodesic at each step. The following result is proved in  \cite[Lemma 3.5]{curvature}.
\begin{lemma}\label{l:decreasing}
For any ample, equiregular geodesic, it holds $d_1 \geq d_2 \geq \ldots \geq d_m$.
\end{lemma}
\begin{definition}\label{def:geodesic dimension}
Given an ample and equiregular geodesic with initial covector $\lam\in T_{x}^{*}M$ we define
\begin{equation}\label{eq:Nlambda}
\calN(\lam):=\sum_{i=1}^m (2i-1) d_i.
\end{equation}
\end{definition}

Fix an ample and equiregular geodesic $\gamma:[0,T]\to M$ and let $\tanf$ be an admissible extension of its tangent vector. For every vector $v\in\calF^i_{\g(t)}$, consider a smooth vector field $X$ such that $X|_{\g(t)}=v$ and $X|_{\g(s)}\in\calF^i_{\g(s)}$ for every $s\in[0,T]$. The Lie derivative $\calL_{\tanf}$ in the direction of $\tanf$ induces a well defined linear map
\begin{equation}\label{eq:linearissime}
\calL_{\tanf}:\calF^i_{\g(s)} \to \calF^{i+1}_{\g(t)}/\calF^i_{\g(t)},\qquad v\mapsto [\tanf,X]\big|_{\g(t)} \mod \calF^i_{\g(t)}.
\end{equation}
Indeed a direct computation shows that this map does not depend on the admissible extension $\tanf$ and on the extension of $X$, under the equiregularity assumption on $\g$. Moreover the kernel of \eqref{eq:linearissime} is given by $\calF^{i-1}_{\g(t)}$. So one obtains  well-defined  linear surjective maps (denoted by the same symbol)
\begin{equation}\label{eq:calL}
\calL_\tanf:\calF^i_{\g(t)}/\calF^{i-1}_{\g(t)}\to \calF^{i+1}_{\g(t)}/\calF^i_{\g(t)}, \qquad i\geq 1.
\end{equation}
In particular $\calL_\tanf^i:\calD_{\g(t)}\to \calF^{i+1}_{\g(t)}/\calF^i_{\g(t)}$ are surjective linear maps defined on the distribution $\calD_{\g(t)}=\calF^{1}_{\g(t)}$.
\bdeff \label{def:symbol}
Given an ample and equiregular geodesic $\g:[0,T]\to M$ we define its \emph{symbol} at $\g(t)$, denoted by $S_{\g(t)}$, as the pair
\bi
\iii[(i)] the graded vector space: $\mathrm{gr}_{\g(t)}(\mc{F}):=\bigoplus_{i=0}^{m-1} \mc{F}^{i+1}_{\g(t)}/\mc{F}^{i}_{\g(t)}$,
\iii[(ii)] the  family of operators: $\calL^{i}_{\tanf}:\distr_{\g(t)}\to \mc{F}^{i+1}_{\g(t)}/\mc{F}^{i}_{\g(t)}$ for $i\geq 1$,
\ei
where $\tanf$ denotes any admissible extension of $\dot \g$.
\edeff

\begin{remark}
Notice that, for the symbol $(V,L)$ associated with an ample and equiregular geodesic, the family of maps $L=\{L_{i}\}_{i=1}^{m}$ where $L_{i}=\calL^{i}_{T}$ satisfies the factorization property $\ker L_{i}\subset \ker L_{i+1}$ for all $i\geq 1$.
\end{remark}

\section{Young diagram, canonical frame and Jacobi fields} \label{s:cfjf}
In this section we briefly recall how to define the canonical frame that can be associated with any ample and equiregular geodesic, introduced in \cite{lizel}. We follow the approach contained in \cite{curvature,barilaririzzicomp}, where the interested reader can find more details.

For an ample, equiregular geodesic
we can build a tableau $\y$ with $m$ columns of length $d_{i}$, for $i=1,\ldots,m$, as follows:
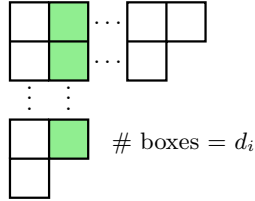
\begin{center}
\begin{tikzpicture}[x=0.26mm, y=0.26mm, inner xsep=0pt, inner ysep=0pt, outer xsep=0pt, outer ysep=0pt]
\path[line width=0mm] (61.05,70.00) rectangle +(127.98,100.00);
\draw(120.00,159.00) node[anchor=base]{\fontsize{9.39}{11.27}\selectfont $\ldots$};
\draw(120.00,139.00) node[anchor=base]{\fontsize{9.39}{11.27}\selectfont $\ldots$};
\draw(80.00,115.00) node[anchor=base]{\fontsize{9.39}{11.27}\selectfont $\vdots$};
\draw(100.00,115.00) node[anchor=base]{\fontsize{9.39}{11.27}\selectfont $\vdots$};
\definecolor{L}{rgb}{0,0,0}
\path[line width=0.30mm, draw=L] (130.00,170.00) -- (130.00,130.00) -- (150.00,130.00) -- (150.00,150.00) -- (170.00,150.00) -- (170.00,170.00) -- cycle;
\path[line width=0.30mm, draw=L] (150.00,170.00) -- (150.00,150.00);
\path[line width=0.30mm, draw=L] (130.00,150.00) -- (150.00,150.00);
\draw(122.00,95.00) node[anchor=base west]{\fontsize{9}{10.24}\selectfont \# boxes = $d_i$};
\definecolor{F}{rgb}{0.565,0.933,0.565}
\path[line width=0.30mm, draw=L, fill=F] (90.00,170.00) [rotate around={270:(90.00,170.00)}] rectangle +(40.00,20.00);
\path[line width=0.30mm, draw=L] (90.00,170.00) -- (70.00,170.00) -- (70.00,130.00) -- (90.00,130.00);
\path[line width=0.30mm, draw=L] (70.00,150.00) -- (110.00,150.00);
\path[line width=0.30mm, draw=L, fill=F] (90.00,110.00) [rotate around={270:(90.00,110.00)}] rectangle +(20.00,20.00);
\path[line width=0.30mm, draw=L] (90.00,70.00) [rotate around={90:(90.00,70.00)}] rectangle +(20.00,20.00);
\path[line width=0.30mm, draw=L] (70.00,90.00) -- (70.00,110.00) -- (90.00,110.00);
\end{tikzpicture}%

\end{center}
The total number of boxes in $\y$ is $n=\dim M=\sum_{i=1}^m d_i$. 

Consider an ample, equiregular geodesic, with Young diagram $\y$, with $k$ rows, and denote the length of the rows by $n_1,\ldots,n_k$. Indeed $n_1+\ldots+n_k = n$. We are going to introduce a moving frame on $T_{\lambda(t)}(T^*M)$ indexed by the boxes of the Young diagram. The notation $ai \in \y$ denotes the generic box of the diagram, where $a=1,\ldots,k$ is the row index, and $i=1,\ldots,n_a$ is the progressive box number, starting from the left, in the specified row. We employ letters $a,b,c,\dots$ for rows, and $i,j,h,\dots$ for the position of the box in the row. 

\begin{figure}[ht]
\centering
\begin{tikzpicture}[x=0.30mm, y=0.30mm, inner xsep=0pt, inner ysep=0pt, outer xsep=0pt, outer ysep=0pt]
\path[line width=0mm] (-75.00,91.40) rectangle +(300.63,71.04);
\definecolor{L}{rgb}{0,0,0}
\definecolor{F}{rgb}{0.565,0.933,0.565}
\path[line width=0.60mm, draw=L, fill=F] (-53.50,91.50) rectangle +(17.92,70.94);
\path[line width=0.15mm, draw=L] (-53.30,127.00) -- (-35.58,127.00);
\path[line width=0.15mm, draw=L] (-53.30,109.29) -- (-35.58,109.29);
\path[line width=0.15mm, draw=L] (-53.30,91.57) -- (-35.58,91.57);
\path[line width=0.15mm, draw=L] (-53.30,144.72) -- (-35.58,144.72);
\path[line width=0.60mm, draw=L, fill=F] (88.32,91.67) rectangle +(17.82,70.76);
\definecolor{F}{rgb}{1,0,0}
\path[line width=0.60mm, draw=L, fill=F] (88.43,144.72) rectangle +(35.43,17.72);
\path[line width=0.60mm, draw=L] (106.14,162.44) -- (106.14,144.72);
\path[line width=0.15mm, draw=L] (88.43,127.00) -- (106.14,127.00);
\path[line width=0.15mm, draw=L] (88.43,109.29) -- (106.14,109.29);
\path[line width=0.15mm, draw=L] (88.43,91.57) -- (106.14,91.57);
\definecolor{F}{rgb}{0.686,0.933,0.933}
\path[line width=0.60mm, draw=L, fill=F] (247.87,91.55) [rotate around={0:(247.87,91.55)}] rectangle +(17.71,17.73);
\definecolor{F}{rgb}{0.565,0.933,0.565}
\path[line width=0.60mm, draw=L, fill=F] (247.87,109.29) rectangle +(35.43,35.43);
\definecolor{F}{rgb}{1,0,0}
\path[line width=0.60mm, draw=L, fill=F] (247.87,144.72) rectangle +(70.86,17.72);
\path[line width=0.60mm, draw=L] (265.58,162.44) -- (265.58,109.29);
\path[line width=0.15mm, draw=L] (247.87,127.00) -- (283.30,127.00);
\path[line width=0.15mm, draw=L] (247.87,91.57) -- (265.58,91.57);
\path[line width=0.60mm, draw=L] (283.30,162.44) -- (283.30,144.72);
\path[line width=0.60mm, draw=L] (301.01,162.44) -- (301.01,144.72);
\draw(-20.85,124.71) node[anchor=base west]{\fontsize{9.39}{11.27}\selectfont level 1};
\path[line width=0.15mm, draw=L] (128.00,162.00) -- (133.00,162.00) -- (133.00,145.00) -- (128.00,145.00);
\path[line width=0.15mm, draw=L] (128.00,145.00) -- (133.00,145.00) -- (133.00,92.00) -- (128.00,92.00);
\draw(139.00,150.00) node[anchor=base west]{\fontsize{9.39}{11.27}\selectfont level 1};
\draw(139.00,116.00) node[anchor=base west]{\fontsize{9.39}{11.27}\selectfont level 2};
\path[line width=0.15mm, draw=L] (324.00,162.00) -- (329.00,162.00) -- (329.00,145.00) -- (324.00,145.00);
\draw(334.00,150.00) node[anchor=base west]{\fontsize{9.39}{11.27}\selectfont level 1};
\path[line width=0.15mm, draw=L] (324.00,145.00) -- (329.00,145.00) -- (329.00,110.00) -- (324.00,110.00);
\draw(334.00,123.00) node[anchor=base west]{\fontsize{9.39}{11.27}\selectfont level 2};
\path[line width=0.15mm, draw=L] (325.00,110.00) -- (329.00,110.00) -- (329.00,92.00) -- (324.00,92.00);
\draw(334.00,98.00) node[anchor=base west]{\fontsize{9.39}{11.27}\selectfont level 3};
\draw(68.00,125.00) node[anchor=base west]{\fontsize{9}{10.24}\selectfont (b)};
\draw(227.00,125.00) node[anchor=base west]{\fontsize{9}{10.24}\selectfont (c)};
\draw(-75.00,124.00) node[anchor=base west]{\fontsize{9}{10.24}\selectfont (a)};
\path[line width=0.15mm, draw=L] (-33.00,162.40) -- (-28.00,162.40) -- (-28.00,91.40) -- (-33.00,91.40);
\end{tikzpicture}%
\caption{Levels (shaded regions) and superboxes (delimited by bold lines) for different Young diagrams: (a) Riemannian, (b) contact, (c) a more general structure.}\label{f:Yd2}
\end{figure}
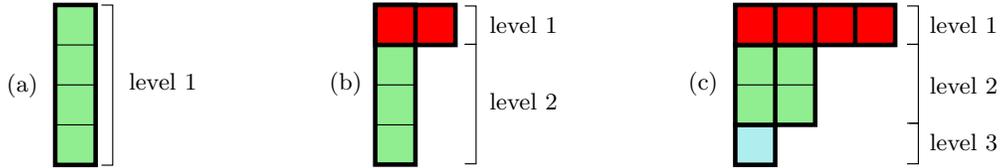
We collect the rows with the same length in $\y$, and we call them \emph{levels} of the Young diagram.  In particular, a level is the union of $r$ rows $\y_1,\ldots,\y_r$, and $r$ is called the \emph{size} of the level. The set of all the boxes $ai \in\y$ that belong to the same column and the same level of $\y$ is called \emph{superbox}. We use Greek letters $\alpha,\beta,\dots$ to denote superboxes. Notice that two boxes $ai$, $bj$ are in the same superbox if and only if $ai$ and $bj$ are in the same column of $\y$ and in possibly distinct rows but with same length, i.e.~if and only if $i=j$ and $n_a = n_b$ (see Fig.~\ref{f:Yd2}).

In what follows, for $V(t)$ a vector field along an integral line $\lambda(t)$ of the Hamiltonian flow, we denote by 
\begin{equation}
\dot{V}(t) := \left.\frac{d}{d\eps}\right|_{\eps=0} e^{-\eps \vec{H}}_* V(t+\eps).
\end{equation}
the Lie derivative of $V$  in the direction of $\vec{H}$. The following theorem is proved in \cite{lizel}.
\begin{theorem}\label{p:can} Assume $\lam(t)$ is the lift of an ample and equiregular geodesic $\g(t)$ with Young diagram $\y$. Then
there exists a smooth moving frame $\{E_{ai},F_{ai}\}_{ai \in \y}$ along $\lambda(t)$ such that
\begin{itemize}
\item[(i)] $\pi_{*}E_{ai}|_{\lambda(t)}=0$.
\item[(ii)] It is a Darboux basis, namely
\[
\sigma(E_{ai},E_{bj}) = \sigma(F_{ai},F_{bj}) = \sigma(E_{ai},F_{bj}) -\delta_{ab}\delta_{ij}=0, \qquad ai,bj \in \y.
\]
\item[(iii)] The frame satisfies the structural equations
\begin{equation}\label{zelframe}
\displaystyle\begin{cases}	
\dot{E}_{ai} = E_{a(i-1)} & a = 1,\dots,k,\quad i = 2,\dots, n_a,\\[0.1cm]
\dot{E}_{a1} = -F_{a1} & a= 1,\dots,k, \\[0.1cm]
\dot{F}_{ai} = \sum_{bj \in \y} R_{ai,bj}(t) E_{bj} - F_{a(i+1)} & a=1,\dots,k,\quad i = 1,\dots,n_a-1,\\[0.1cm]
\dot{F}_{an_a} = \sum_{bj \in \y} R_{an_a,bj}(t) E_{bj}  & a = 1, \dots,k,
\end{cases}
\end{equation}
for some smooth family of $n\times n$ symmetric matrices $R(t)$, with components $R_{ai,bj}(t) = R_{bj,ai}(t)$, indexed by the boxes of the Young diagram $\y$. The matrix $R(t)$ is \emph{normal}  in the sense of \cite{lizel}.
\end{itemize}
If $\{\wt{E}_{ai},\wt{F}_{ai}\}_{ai \in \y}$ is another smooth moving frame along $\lambda(t)$ satisfying (i)-(iii), with some normal matrix $\wt{R}(t)$, then for any superbox $\alpha$ of size $r$ there exists an orthogonal constant $r\times r$ matrix $O^\alpha$ such that
\begin{equation}
\wt{E}_{ai} = \sum_{bj \in \alpha} O^\alpha_{ai,bj} E_{bj}, \qquad \wt{F}_{ai} = \sum_{bj \in \alpha} O^\alpha_{ai,bj} F_{bj}, \qquad ai \in \alpha.
\end{equation}
\end{theorem}
The explicit condition for the matrix $R(t)$ to be \emph{normal} can be found in \cite[Appendix F]{curvature} (cf.\ also the original definition in \cite{lizel}).
\begin{remark}\label{rmk:notation}
For $a=1,\dots,k$, we denote by $E_a$ the $n_a$-dimensional row vector
$
E_a = (E_{a1},\dots,E_{an_a}),
$
with analogous notation for $F_a$. Denote then by $E$  the $n$-dimensional row vector
$
E = (E_1,\dots,E_k),
$
and similarly for $F$. Then, we rewrite the system \eqref{zelframe} as follows 
\begin{equation}\label{eq:Jacobiframe2}
\begin{pmatrix}
\dot{E}^* \\
\dot{F}^*
\end{pmatrix} = 
\begin{pmatrix}
C^*_1 & -C_2 \\
R(t) & -C_1
\end{pmatrix} \begin{pmatrix}
E^*\\
F^*
\end{pmatrix},
\end{equation}
where $C_1 = C_1(\y)$, $C_2=C_2(\y)$ are $n\times n$ matrices, depending on the Young diagram $\y$, defined as follows:
 for $a,b = 1,\dots,k$, $i=1,\dots,n_a$, $j=1,\dots,n_b$:
\begin{equation}
[C_1]_{ai,bj} := \delta_{ab}\delta_{i,j-1}, \label{eq:G1},\qquad
[C_2]_{ai,bj} := \delta_{ab}\delta_{i1}\delta_{j1}. 
\end{equation}
It is convenient to see $C_1$ and $C_2$ as block diagonal matrices:
\begin{equation}
C_i(\y) := \begin{pmatrix} 
C_i(\y_1) &  &    \\
 &  \ddots & \\
 &   & C_i(\y_k)
\end{pmatrix}, \qquad i =1,2,
\end{equation}
the $a$-th block being the $n_a\times n_a$ matrices
\begin{equation}
C_1(\y_a) := \begin{pmatrix}
0 & \mathbb{I}_{n_a-1} \\
0 & 0
\end{pmatrix} , 
\qquad C_2(\y_a) := \begin{pmatrix}
1 & 0 \\
0 & 0_{n_a-1}
\end{pmatrix},
\end{equation}
where $\mathbb{I}_{m}$ is the $m \times m$ identity matrix and $0_{m}$ is the $m \times m$ zero matrix. 
\end{remark}

\subsection{The Jacobi equation}
A vector field $\J(t)$ along $\lam(t)$ is called a \emph{Jacobi field} if it satisfies 
\begin{equation}\label{eq:defJF}
\dot{\J} = 0.
\end{equation}
The space of solutions of \eqref{eq:defJF} is a $2n$-dimensional vector space. The projections $J=\pi_{*}\J$ are vector fields on $M$ corresponding to one-parameter variations of $\g(t)=\pi(\lam(t))$ through geodesics; in the Riemannian case (without drift field) they coincide with the classical Jacobi fields.

We intend to write \eqref{eq:defJF} using the natural symplectic structure $\sigma$ of $T^{*}M$ and the canonical frame.  First, observe that on $T^*M$ there is a natural smooth sub-bundle of Lagrangian\footnote{A Lagrangian subspace $L \subset \Sigma$ of a symplectic vector space $(\Sigma,\sigma)$ is a subspace with $2\dim L = \dim\Sigma$ and $\sigma|_{L} = 0$.} spaces:
\begin{equation}
\ver_{\lambda} := \ker \pi_*|_{\lambda} = T_\lambda(T^*_{\pi(\lambda)} M).
\end{equation}
We call this the \emph{vertical subspace}. Then, let $\{E_i(t),F_i(t)\}_{i=1}^{n}$ be a canonical frame along $\lambda(t)$.  The fields $E_1,\ldots,E_n$ belong to the vertical subspace. In terms of this frame, $\J(t)$ has components $(p(t),x(t)) \in \R^{2n}$:
\begin{equation}
\J(t) = \sum_{i=1}^n p_{i}(t) E_{i}(t) + x_{i}(t) F_{i}(t).
\end{equation}
In turn, the Jacobi equation, written in terms of the components $(p(t),x(t))$, becomes
\begin{equation}
\begin{pmatrix}\label{eq:Jacobicoord}
\dot{p} \\ \dot{x}
\end{pmatrix} = \begin{pmatrix} - C_1 & -R(t) \\ C_2 & C_1^{*}
\end{pmatrix} \begin{pmatrix}
p \\ x
\end{pmatrix}.
\end{equation}
%
This is a generalization of the classical Jacobi equation seen as first-order equation for fields on the cotangent bundle. Its  structure depends on the Young diagram of the geodesic through the matrices $C_{i}(\y)$, while the remaining  invariants are contained in the curvature matrix $R(t)$. Notice that this includes the Riemannian case, where $\y$ is the same for every geodesic, with $C_{1}=0$ and $C_{2}=\mathbb{I}$.

\subsection{Geodesic cost and curvature operator}
In this section we define the geodesic cost and the curvature operator associated with a geodesic $\gamma$. This operator generalizes the Riemannian sectional curvature operator.

\begin{definition}
Let $x_0\in M$ and consider an ample geodesic $\gamma$ such that $\gamma(0)=x_0$. The \emph{geodesic cost} associated with $\gamma$ is the family of functions
$$c_t(x):=-S_t(x,\gamma(t)),\qquad x\in M, t>0,$$
where $S_{t}$ is the value function defined in \eqref{eq:value}.
\end{definition}
\begin{figure}
\begin{center}
\scalebox{0.8} 
{
\begin{pspicture}(0,-1.6)(6.42,1.6)
\definecolor{colour0}{rgb}{0.0,0.4,1.0}
\definecolor{colour1}{rgb}{0.0,0.6,1.0}
\psdots[linecolor=black, dotsize=0.12](0.8,-0.8)
\psbezier[linecolor=black, linewidth=0.04](0.8,-0.8)(2.8,-0.4)(6.4,0.8)(6.4,1.6)
\psdots[linecolor=colour0, dotsize=0.12](4.8,0.4)
\psbezier[linecolor=colour1, linewidth=0.04, linestyle=dashed, dash=0.17638889cm 0.10583334cm](4.8,0.4)(3.2,0.8)(1.6,0.0)(0.8,-0.4)
\psdots[linecolor=colour1, dotsize=0.12](0.8,-0.4)
\pscircle[linecolor=black, linewidth=0.02, linestyle=dotted, dotsep=0.10583334cm, dimen=outer](0.8,-0.8){0.8}
\rput[bl](0.4,-1.2){$x_0$}
\rput[bl](4.8,0.0){$\gamma(t)$}
\rput[bl](0.4,-0.4){$x$}
\rput[bl](2.4,-1.6){$x\mapsto -S_t(x,\gamma(t))$}
\end{pspicture}
}
\caption{The geodesic cost function}%
\label{fig:geodesic cost}%
\end{center}
\end{figure}
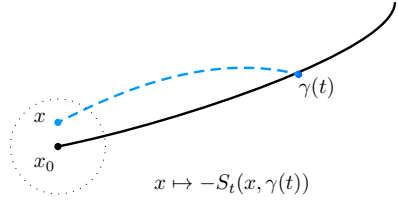
By \cite[Theorem 4.2]{curvature}, given an ample curve $\gamma(t)=\pi(e^{t\vec{H}}(\lambda))$ starting at $x_{0}$, the geodesic cost function $c_{t}(x)$ is smooth in a neighborhood of $x_0$ and for $t>0$ sufficiently small. Moreover the differential $d_{x_0}c_t=\lambda$ for every $t$ small. 

Let $\dot{c}_t=\frac{\partial}{\partial t}c_{t}$ denote the derivative with respect to $t$ of the geodesic cost. Then $\dot{c}_t$ has a critical point at $x_0$ and its second differential $d^2_{x_0}\dot{c}_t:T_{x_0}M\to\bbR$ is defined as $$d^2_{x_0}\dot{c}_t(v)=\left.\frac{d^2}{ds^2}\right|_{s=0}\dot{c}_t(\gamma(s)), \qquad \gamma(0)=x_0,\quad \dot{\gamma}(0)=v.$$
We restrict the second differential of $\dot{c}_t$ to the distribution $\calD_{x_0}$ and we define the following family of symmetric operators $\calQ_{\lambda}(t):\calD_{x_0}\to\calD_{x_0}$, for small $t$, associated with $d^2_{x_0}\dot{c}_t$ through the inner product defined on $\calD_{x_{0}}$:
\begin{equation}
\label{eq:volume Q}
d^2_{x_0}\dot{c}_t(v):=\la \calQ_{\lambda}(t)v|v\ra_{x_0}, \qquad t>0, v\in\calD_{x_0}.
\end{equation}
The following result is contained in \cite[Theorem A]{curvature}. 
\begin{theorem}
Let $\g:[0,T]\to M$ be an ample geodesic with initial covector $\lambda\in T_{x_0}^*M$ and let $\calQ_{\lambda}(t):\calD_{x_0}\to\calD_{x_0}$ be defined by \eqref{eq:volume Q}. Then $t\mapsto t^2\calQ_{\lambda}(t)$ can be extended to a smooth family of symmetric operators on $\calD_{x_0}$ for small $t\geq 0$. Moreover
$$\calI_{\lambda}:=\lim_{t\searrow 0}t^2\calQ_{\lambda}(t)\geq \bbI>0, \qquad \left.\frac{d}{dt}\right|_{t=0}t^2\calQ_{\lambda}(t)=0,$$
where $\bbI$ is the identity operator. In particular, there exists a symmetric operator $\calR_{\lambda}:\calD_{x_0}\to\calD_{x_0}$ such that
\begin{equation}\label{eq:R}
\calQ_{\lambda}(t)=\frac{1}{t^2}\calI_{\lambda}+\frac{1}{3}\calR_{\lambda}+O(t),\qquad t>0.
\end{equation}
\end{theorem}
\begin{definition}
We call the symmetric operator $\calR_{\lambda}:\calD_{x_0}\to\calD_{x_0}$ in \eqref{eq:R} the \emph{curvature} at $\lambda$. Its trace $\text{tr}\,\calR_{\lambda}$ is the \emph{Ricci curvature} at $\lambda$. 
\end{definition}

When $\g$ is also equiregular, the curvature operator $\calR_{\lambda}$ can be written in terms of the smooth $n$-dimensional symmetric matrix $R(t)$ introduced in the canonical equations \eqref{zelframe}, as we now describe.
 
Let $\g(t)=\pi (e^{t\vec{H}}(\lambda))$ be ample and equiregular, and let $\{E_{ai}(t),F_{ai}(t)\}_{ai\in D}$ be a canonical frame along the curve $\lambda(t)$.
\begin{lemma}[\cite{curvature}, Lemma 8.3]\label{lemma:adapted basis}
The set of vector fields along $\g(t)$
$$X_{ai}(t):=\pi_* F_{ai}(t),\qquad ai\in D$$
is a basis for $T_{\g(t)}M$ adapted to the flag $\{\calF^i_{\g(t)}\}_{i=1}^m$ and $\{X_{a1}(t)\}_{a=1}^k$ is an orthonormal basis for $\calD_{\g(t)}$ along the geodesic.
\end{lemma}
The following proposition is proved in \cite[Section 7.4]{curvature}.
\begin{prop}Let $\g$ be an ample and equiregular geodesic with initial covector $\lambda$. 
The matrix representing $\calR_{\lambda}$ in terms of the orthonormal basis $\{X_{a1}(t)\}_{a=1}^k$ depends only on the elements of $R_{a1,b1}(0)$ corresponding to the first column of the associated Young diagram. More precisely we have
\begin{equation}\label{eq:ROmega}
(\calR_{\lambda})_{ab}=3\Omega(n_a,n_b)R_{a1,b1}(0), \qquad a,b=1,\ldots,k,
\end{equation}
where for $i,j\in\bbN$ we set
$$\Omega(i,j)=
\begin{cases}
0, & |i-j|\geq 2,\\
\dfrac{1}{4(i+j)}, & |i-j|=1,\\
\dfrac{i}{4i^2-1}, & i=j.
\end{cases}$$
\end{prop}

\section{Volume geodesic derivative} \label{s:rho}
In this section we introduce the volume geodesic derivative $\rho$ describing the interaction between the dynamics and the volume $\mu$ on the manifold, and we study its basic properties. 


Recall that, given a smooth volume form $\mu$ on $M$, its value $\mu_{x}$ at a point is a nonzero element of the space $\Lambda^{n}(T_{x}M)$. We can associate with it the unique element $\mu_{x}^{*}$ in $\Lambda^{n}(T_{x}M)^{*}=\Lambda^{n}(T^{*}_{x}M)$ satisfying $\mu_{x}^{*}(\mu_{x})=1$. This defines a volume form on the fiber $T^{*}_{x}M$. By the canonical identification $T^{*}_{x}M\simeq T_{\lam}(T^{*}_{x}M)$ of a vector space with its tangent space to a point, this induces a volume form $\mu^{*}_{\lam}$ on the vertical space $\V_{\lambda}:=T_{\lam}(T^{*}_{x}M)$ for each $\lam\in T_{x}^{*}M$.

Let $\gamma(t)=\pi(\lam(t))$ be an ample and equiregular geodesic defined on $[0,T]$, with $\lam(t)=e^{t\vec H}(\lam) $ and  $\lambda\in T^{*}_{x}M$. Denote by $\mathcal{A}$ the set of $\lambda\in T^{*}M$ such that the corresponding trajectory is ample and equiregular. For a fixed $x\in M$, we set $\mathcal{A}_{x}:=\mathcal{A}\cap T^{*}_{x}M$.

Notice that, if $\lam\in \mathcal{A}_{x}$, then the exponential map $\pi\circ e^{t\vec H}:T^{*}_{x}M\to M$ is a local diffeomorphism at $\lambda$, for small $t\geq 0$. Then it makes sense to consider the pull-back measure $(\pi\circ e^{t\vec H})^{*}\mu$ and compare its restriction to the  vertical space $\V_{\lambda}$ with $\mu^{*}_{\lambda}$.

\begin{definition} For every $\lam\in \mathcal{A}_{x}$ we define the volume geodesic derivative by the identity 
\bqn
\rho(\lam)\mu^{*}_{\lam}=\frac{d}{dt}\bigg|_{t=0} \log \left(t^{-\mathcal{N}(\lambda)}\left.(\pi\circ e^{t\vec H})^{*}\mu\right|_{\V_{\lambda}}\right),
\eqn
where $\calN(\lambda)$ is defined in \eqref{eq:Nlambda}.
\end{definition}

Let $\{\theta_{ai}(t)\}_{ai\in D}\in T_{\g(t)}^*M$ be the coframe dual to $X_{ai}(t)=\pi_{*}F_{ai}(t)$ and define a volume form $\omega$ along $\g$ as
\begin{equation}\label{eq:volumeomega}
\omega_{\g(t)}:=\theta_{11}(t)\wedge \theta_{12}(t)\wedge\ldots\wedge\theta_{k n_k}(t).
\end{equation}
Given a fixed smooth volume $\mu$ on $M$, let $g_\lambda:[0,T]\to \bbR$ be the smooth function such that
\begin{equation}
\mu_{\gamma(t)}=e^{g_\lambda(t)}\omega_{\gamma(t)}.
\label{eq:g}
\end{equation}
The first main result of this section is the relation between the invariant $\rho$ and the function $g_\lambda(t)$ just introduced.
\begin{prop} For every $\lambda\in \mathcal{A}_{x}$ one has $\rho(\lam)=\dot{g}_{\lam}(0)$.
\end{prop}
\noindent The proof of this Proposition is a corollary of the proof of the main theorem, that is proved in Section \ref{s:proof}. We exploit the previous identity to prove some useful properties of the invariant $\rho$. We start by the following lemma.
\begin{lemma} \label{lemma:g}
Let $\g(t)=\pi (e^{t\vec{H}}(\lambda))$ be an ample and equiregular geodesic. Then we have
$$\dot g_{\lam}(t)=\dot g_{\lam(t)}(0), \qquad \forall t\in[0,T].$$
\end{lemma}
\begin{proof}
Let $\lambda(t)=e^{t\vec{H}}(\lambda)\in T^*M$ be the lifted extremal and denote by $\gamma_t(s):=\g(t+s)$. Then $\g_t(s)=\pi (e^{s\vec{H}}(\lam(t)))$ and we have the sequence of identities
$$e^{g_\lambda(t+s)}\omega_{\g(t+s)}=\mu_{\g(t+s)}=\mu_{\g_t(s)}=e^{g_{\lambda(t)}(s)}\omega_{\g_t(s)}$$ 
Moreover $\omega_{\g(t+s)}=\omega_{\g_t(s)}$ since, if $(E_{\lambda(t+s)},F_{\lambda(t+s)})$ is a canonical frame along $\lam(t+s)$, it is a canonical frame also for $e^{s\vec{H}}\lam(t)$. It follows that $g_\lambda(t+s)=g_{\lambda(t)}(s)$ for every $s$ and differentiating with respcet to $s$ at $s=0$ one gets the result.
\end{proof}
%

Lemma \ref{lemma:g} allows us to write $g$ as a function of $\rho$, as follows
\begin{equation}\label{eq:gintegral}
g_{\lam}(t)=g_{\lam}(0)+\int_{0}^{t}\dot g_{\lam}(s) ds=g_{\lam}(0)+\int_{0}^{t} \rho(\lam(s)) ds.
\end{equation}

\begin{prop} Let $\tanf$ be any admissible extension of $\dot{\g}$ and $\omega$ the $n$-form defined in \eqref{eq:volumeomega}. Then for every $\lambda\in T^*_x M$
\begin{equation}\label{eq:rhodiv}
\rho(\lambda)= (\mathrm{div}_{\mu} \tanf-\mathrm{div}_{\omega} \tanf)|_{x}.
\end{equation}
\end{prop}
\begin{proof} It is a direct consequence of  the classical identity 
$$\diver_{f\omega}X-\diver_\omega X=X(\log f)$$ which holds for every smooth volume form $\omega$, smooth function $f$ and smooth vector field $X$.
\end{proof}
\begin{remark}[On the volume form $\omega$, I]\label{r:rem}
In the Riemannian case $\{X_{ai}(t)\}_{ai\in D}$ is an orthonormal frame for the Riemannian metric by Lemma \ref{lemma:adapted basis} and the form $\omega$ coincides with the restriction of the canonical Riemannian volume $\vol_g$ on the curve $\gamma$. Hence
\begin{equation}\label{eq:rho000}
\rho(v)=(\mathrm{div}_{\mu}\tanf) |_x-(\mathrm{div}_{\vol_g}\tanf) |_x.
\end{equation}
In the general case $\rho$  can still be represented as the difference of two divergences
but the volume form $\omega$ \emph{depends on the curve} $\gamma$ and is not the restriction to the curve of a global volume form.
\end{remark}

Next we recall a refinement of Lemma \ref{lemma:adapted basis}.
\begin{lemma}[\cite{curvature}, Lemma 8.5]\label{lemma:Xai}
For $t\in[0,T]$, the projections $X_{ai}(t)=\pi_{*}F_{ai}(t)$ satisfy
$$X_{ai}(t)=(-1)^{i-1}\calL_{\tanf}^{i-1}(X_{a1}(t))\;\mod \mc{F}^{i-1}_{\g(t)},\qquad a=1,\ldots,k,\ i=1,\ldots,n_a.$$
\end{lemma}

\begin{prop}[$\rho$ depends only on $\mu$ and the symbol along $\g$] \label{p:pppp}
Let $\g,\g'$ be two geodesics associated with initial covectors $\lam\in \mathcal{A}_{\gamma(0)}$ and $\lam'\in \mathcal{A}_{\gamma'(0)}$ respectively.
Assume that there exists a diffeomorphism $\phi$ on $M$ such that for $t\geq 0$ small enough
\bi
\iii[(i)] $\phi(\gamma(t))=\gamma'(t)$, 
\iii[(ii)] $\phi_{*}|_{\gamma(t)}$ induces an isomorphism of symbols between $S_{\g(t)}$ and $S_{\g'(t)}$
\iii[(iii)] $\phi^{*}\mu_{\gamma'(t)}=\mu_{\gamma(t)}$, 
\ei
Then $\rho(\lambda)=\rho (\lambda')$.
\end{prop}
\begin{proof} Let $\{E_{ai}(t),F_{ai}(t)\}_{ai\in D}$ and $\{E'_{ai}(t),F'_{ai}(t)\}_{ai\in D}$  be canonical frames with respect to $\lambda$ and $\lambda'$ respectively, and $X_{ai}(t)=\pi_* (F_{ai}(t))$, $X_{ai}'(t)=\pi_*(F'_{ai}(t))$ be the associated basis of $T_{\g(t)}M$ and $T_{\g'(t)}M$. Since $\omega$ evaluated on the projection of the canonical frame gives 1 by construction, we have
\begin{gather} e^{g_\lambda(t)}=|\mu_{\g(t)}(X_{11}(t),\ldots,X_{kn_k}(t))|, \\
 e^{g_{\lambda'}(t)}=|\mu_{\g'(t)}(X'_{11}(t),\ldots,X'_{kn_k}(t))|.
 \end{gather}
Recall that $\{X_{a1}\}_{a=1}^k$ (resp.\ $\{X'_{a1}\}_{a=1}^k$) is an orthonormal basis for $\calD_{\g(t)}$ (resp.\ $\calD'_{\g(t)}$). Since the linear map $\phi_*|_{\gamma(t)}:\calD_{\g(t)}\to \calD'_{\g(t)}$ is an isometry for small $t\geq0$, there exists a family of orthogonal $k\times k$ matrices $O(t)$ such that
$$X'_{a1}(t)=\sum_{b=1}^k O_{ab}(t)\phi_*(X_{b1}(t)), \qquad \mbox{ for } a=1,\ldots,k.$$
Moreover using Lemma \ref{lemma:Xai} we have for $i>1$
\begin{equation}
\begin{split}
X'_{ai}(t)=&(-1)^{i-1}\calL^{(i-1)}_{\tanf'}(X'_{a1}(t))\mod\mc{F}^{i-1}_{\g'(t)}\\
=&(-1)^{i-1}\calL^{(i-1)}_{\tanf'}\left(\sum_{b=1}^k O(t)_{ab}\phi_*(X_{b1}(t))\right)\mod\mc{F}^{i-1}_{\g'(t)}\\
=&(-1)^{i-1}\sum_{b=1}^k O(t)_{ab}\calL^{(i-1)}_{\tanf'}\left(\phi_*(X_{b1}(t))\right)\mod\,\mc{F}^{i-1}_{\g'(t)},
\end{split}
\end{equation}
where the last identity follows by the chain rule. Indeed, when one differentiates the matrix $O(t)$, one obtains elements of $\mc{F}^{i-1}_{\g'(t)}$. Then
\begin{equation}
\begin{split}
X'_{ai}(t)=&(-1)^{i-1}\sum_{b=1}^k O(t)_{ab}\phi_*\calL^{(i-1)}_{\tanf}(X_{b1}(t))\,\mod\mc{F}^{i-1}_{\g'(t)}\\
=&\sum_{b=1}^{d_i} O(t)_{ab}\phi_*X_{bi}(t)\,\mod\mc{F}^{i-1}_{\g'(t)},
\end{split}
\end{equation}
where the sum is restricted to those indices $b$ such that $bi\in \mathrm{D}$. This proves that there exists an orthogonal transformation that sends $\phi_*X_{ai}$ in $X'_{ai}$. Therefore
\begin{equation}
\begin{split}
e^{g_{\lambda'}(t)}=&\left| \mu_{\g'(t)}\left(X'_{1,1}(t),\ldots, X'_{kn_k}\right)\right|=\left| \mu_{\g'(t)}\left(\phi_*X_{1,1}(t),\ldots, \phi_*X_{kn_k}\right)\right|\\
=&\left| (\phi^*\mu)_{\g(t)}\left(X_{1,1}(t),\ldots, X_{kn_k}\right)\right|=e^{g_\lambda(t)},
\end{split}
\end{equation}
and the proof is complete.
\end{proof}
Actually, from the previous proof it follows that the invariant $\rho$ depends only on the 1-jet of the one-parameter family of symbols (and the volume form $\mu$) along the geodesics.

\begin{remark}[On the volume form $\omega$, II]
The volume form $\omega$ depends only on the symbol of the structure along the geodesic, that represents the microlocal nilpotent approximation of the structure at $x$ along $\g(t)$. Symbols at different points along any geodesic in the Riemannian manifold are isomorphic, while in the general case this symbol could depend on the point on the curve. This is analogous to what happens for the nilpotent approximation for a distribution (see for instance the discussion contained in \cite{ABBH}). 
\end{remark}

\begin{lemma} 
Let $\g(t)=\pi( \lambda(t))$ be an ample and equiregular geodesic. Assume that $\tanf$ is an admissible extension of its tangent vector such that $e^{t\tanf}$ is an isometry of the distribution along $\g(t)$. Then $\mathrm{div}_{\omega} \tanf\big|_{\g(t)}=0$ and $\rho(\lambda(t))=\mathrm{div}_{\mu}\tanf$.
\end{lemma}
\begin{proof}
If we show that $\mathrm{div}_{\omega} \tanf_{\g(t)}=0$, then from   \eqref{eq:rhodiv} it immediately follows that $\rho(\lambda(t))=\mathrm{div}_{\mu}\tanf$ and $\rho$ depends only on the variation of the volume $\mu$ along the curve.

Let $\{X_{ai}(t)\}_{ai\in D}$ be the basis of $T_{\g(t)}M$ induced by the canonical frame along the curve $\lambda(t)$. The divergence is computed as
\begin{equation}
\begin{split}
(\mathrm{div}_{\omega} \tanf)\big|_{\g(t)}\omega_{\g(t)}(X_{11}(t),\ldots,X_{kn_k}(t))=&\calL_{\tanf}\omega(X_{11},\ldots,X_{kn_k})\big|_{\g(t)}\\
=&\left.\frac{d}{d\epsilon}\right|_{\epsilon=0}\omega_{\g(t+\epsilon)}(e^{\epsilon \tanf}_*X_{11},\ldots,e^{\epsilon \tanf}_*X_{kn_k}).
\end{split}
\end{equation}
Since the flow of $\tanf$ is an isometry of the graded structure that defines the symbol, the last quantity is equal to 0,
which proves that $\mathrm{div}_{\omega}\tanf=0$ along the curve. 
\end{proof}

\begin{lemma}
The function $\rho:\mathcal{A}\to \R$ is a rational function.
\end{lemma}
\begin{proof} Since $H$ is a quadratic function on fibers, then the vector field $\vec{H}$ is fiber-wise polynomial. Therefore for any vector field $V(t)\in T_{\lambda(t)}(T^*M)$, the quantity $\dot{V}=[\vec{H},V]$ is a rational function of the initial covector $\lambda$. It follows that both $E$ and $F$ are rational as functions of $\lambda$, and so are also the projections $X(t)=\pi_* F(t)$. We conclude that 
$$e^{g_\lambda(t)}=\left|\mu_{\g(t)}(X_{a1}(t),\ldots,X_{kn_k}(t))\right|,$$
and  the coefficients of its Taylor expansion, are rational expressions in $\lambda$.
\end{proof}

\begin{remark}
If the symbol is constant along the trajectory (i.e., symbols at different points are isomorphic) through a diffeomorphism $\phi$ and $\mu$ is preserved by $\phi$, then $\rho(\lambda(t))=0$.

Indeed it is sufficient to apply Proposition \ref{p:pppp}  to $\gamma_t(s):=\g(t+s)$ for every $s$ and one gets for $s,t\geq 0$ small $g(t)=g(t+s)$, that means that $g$ is constant and $\rho=0$.
\end{remark}

\bdeff \label{d:unimod} We say that an Hamiltonian of our class \eqref{eq:classH} is \emph{unimodular} if there exists a volume form $\mu$ such that $\rho=0$ on $\mathcal{A}$.
\edeff
It is easy to see that if an Hamiltonian is unimodular with respect to some volume form $\mu$, then $\mu$ is unique.
\section{A formula for $\rho$} \label{s:formula}
In this section we provide a formula to compute $\rho$ in terms only of the volume $\mu$ and the linear maps $\calL_\tanf^i$. This will give another proof of the fact that the quantity $\rho(\lambda(t))$ depends only on the symbol and on $\mu$ along  $\g(t)=\pi(\lambda(t))$.

Fix a smooth volume $\mu$ on $M$ and let $Y_1,\ldots,Y_k$ be an orthonormal basis of $\calD$ in a neighborhood of $x_0$. Choose vector fields $Y_{k+1},\ldots,Y_n$ such that $Y_1,\ldots,Y_n$ is a local basis satisfying $\mu(Y_1,\ldots,Y_n)=1$ and define an auxiliary inner product on the tangent space declaring that this basis is orthonormal. 

Let $\gamma(t)=\pi( e^{t\vec{H}}(\lambda))$ be an ample and equiregular curve, with initial covector $\lambda\in T^*_{x_0}M$. Recall that, according to the definition of $g_\lambda(t)$, it holds
\begin{equation}\label{eq:linearterm}
g_\lambda(t) = \log|\mu(P_t)|,
\end{equation}
where $P_t$ is the parallelotope whose edges are the projections $\{X_{ai}(t)\}_{ai \in D}$ of the horizontal part of the canonical frame $X_{ai}(t) = \pi_* \circ e^{t\vec{H}}_* F_{ai}(t) \in T_{\gamma(t)} M$, namely
\begin{equation}\label{eq:parallelotope}
P_t=\bigwedge_{ai \in D} X_{ai}(t).
\end{equation}
By Lemma~\ref{lemma:Xai} we can write the adapted frame $\{X_{ai}\}_{ai \in D}$ in terms of the smooth linear maps $\mc{L}_\tanf$, and we obtain the following identity
\begin{equation}\label{eq:parallelotope2}
P_t = \bigwedge_{i=1}^m \bigwedge_{a_i = 1}^{d_i} X_{a_i i}(t) = \bigwedge_{i=1}^m \bigwedge_{a_i = 1}^{d_i} \mc{L}^{i-1}_\tanf( X_{a_i 1}(t)).
\end{equation}

Consider the flag $\{\calF_{\g(t)}^i\}_{i=1}^m$ and, using the auxiliary inner product induced by the choice of the basis, define the following sequence of subspaces of $T_{\g(t)}M$: for every $i\geq 1$ set  (with the understanding that $\calF^{0}=\{0\}$)
\begin{equation}
V_i:=\calF^i_{\g(t)}\cap (\calF^{i-1}_{\g(t)})^{\perp}.
\end{equation}
The subspace $V_{i}$ has dimension $\dim V_{i}=\dim \calF^i_{\g(t)}-\dim \calF^{i-1}_{\g(t)}$. Therefore there exists an isomorphism between $\calF^i_{\g(t)}/\calF^{i-1}_{\g(t)}$ and $V_i$, such that every $Y\in \calF^i_{\g(t)}/\calF^{i-1}_{\g(t)}$ is associated with the  element  of its equivalent class that lies in $V_i$. In conclusion, for the computation of $g_\lambda(t)$ in \eqref{eq:linearterm}, one can replace the vector $\mc{L}^{i-1}_\tanf( X_{a_i 1}(t))$ of the parallelotope in \eqref{eq:parallelotope2} with the corresponding equivalent element in $V_i$.


Now consider the surjective map $\calL_{\tanf}^{i-1}:\calD_{\g(t)}\to \calF^{i}_{\g(t)}/\calF^{i-1}_{\g(t)}$. For every $i=1,\ldots, m$ this map descends to an isomorphism $\calL_{\tanf}^{i-1}:\calD_{\g(t)}/{\ker \calL_\tanf^{i-1}}\to \calF^{i}_{\g(t)}/\calF^{i-1}_{\g(t)} \simeq V_{i}$. Then, thanks to the inner product structure on $V_{i}$, we can consider the map 
$$(\calL_\tanf^{i-1})^* \circ \calL_\tanf^{i-1} : \calD_{\g(t)}/{\ker \calL_\tanf^{i-1}} \to \calD_{\g(t)}/{\ker \calL_\tanf^{i-1}}$$
 obtained by composing $\calL_\tanf^{i-1}$ with its adjoint $(\calL_\tanf^{i-1})^*$. This composition is a symmetric invertible operator and we define the smooth family of symmetric operators
\begin{equation}\label{eq:Mt}
M_i(t) :=  (\mc{L}_\tanf^{i-1})^* \circ \mc{L}_\tanf^{i-1}: \distr_{\gamma(t)}/ \ker \mc{L}_\tanf^{i-1} \to \distr_{\gamma(t)}/ \ker \mc{L}_\tanf^{i-1}, \qquad i=1,\ldots,m.
\end{equation}
Recall in particular that for every $v_1,v_2\in \distr_{\gamma(t)}/ \ker \mc{L}_\tanf^{i-1}$ it holds the identity
$$\langle(\mc{L}_\tanf^{i-1})^* \circ \mc{L}_\tanf^{i-1} v_1,v_2\rangle_{ \distr_{\gamma(t)}}=\langle \mc{L}_\tanf^{i-1} v_1,\mc{L}_\tanf^{i-1}v_2\rangle_{ V_i}.$$ 

By the expression of the parallelotope $P_t$ with elements of the subspaces $V_i$ and the definition of $\mu$ as the dual of an orthonormal basis of $T_{\g(t)}M$, we have
\begin{equation}
|\mu(P_t)| =\left|\mu\left(\bigwedge_{i=1}^m \bigwedge_{a_i = 1}^{d_i} \mc{L}^{i-1}_\tanf( X_{a_i 1}(t))\right)\right|= \sqrt{\prod_{i=1}^m \det M_i(t)}.
\end{equation}

This formula does not depend on the chosen extension $Y_{k+1},\ldots,Y_n$ of the orthonormal basis of $\calD$, since in the computations we only used that the volume $\mu$ evaluated at this basis is equal to $1$.
For $\rho(\lambda)=\left.\frac{d}{dt}\right|_{t=0} \log |\mu(P_t)|$, a simple computation shows that
\begin{equation}\label{eq:MMM}
\rho(\lambda)=\frac{1}{2}\sum_{i=1}^m \tr\left(M_i(0)^{-1}\dot{M}_i(0)\right).
\end{equation}

\section{Sub-Riemannian manifolds} \label{s:sr}
In this section we specialize our construction to sub-Riemannian manifolds and we investigate in more details the properties of the invariant $\rho$ for these structures. 

Recall that a sub-Riemannian structure on a smooth manifold $M$ is given by a completely non-integrable vector distribution $\distr$ endowed with an inner product on it. In particular $\distr$ has constant rank but does not need to be equiregular. An \emph{admissible curve} is a curve that is almost everywhere tangent to $\distr$ and for such a curve $\g$ we can compute its length by the classical formula
$$\ell(\g)=\int_{0}^{T}\|\dot\gamma(s)\|ds.$$ 
Once we fix a local orthonormal frame $X_1,\ldots,X_k$ for $g$ on $\distr$, the problem of finding the geodesics in a sub-Riemannian manifold, namely the problem of minimizing the length of a curve between two fixed points, is equivalent to the minimization of the energy (with $T>0$ fixed) and can be rewritten as the control problem
\begin{equation}
\left\{
\begin{array}{l}
\dot{x}=\sum_{i=1}^k u_i X_i(x)\\
J_{T}(u)=\frac{1}{2}\int_0^T \|u(s)\|^2 ds \leadsto\min,
\end{array}\right.
\end{equation}
This is an affine control problem, with zero drift field and quadratic cost. The complete non-integrability assumption on the distribution $\distr=\text{span}\{X_1,\ldots,X_k\}$ implies that the assumptions (H0)-(H1) are satisfied. 
The Hamiltonian function is fiber-wise quadratic and convex on fibers. In coordinates it is written as
$$H(p,x)=\frac{1}{2}\sum_{i=1}^k \la p, X_i(x)\ra^2,\qquad (p,x)\in T^{*}M.$$

Denote by $\mathcal{A}$ the set of $\lambda\in T^{*}M$ such that the corresponding trajectory is ample and equiregular. For a fixed $x\in M$ we set $\mathcal{A}_{x}=\mathcal{A}\cap T^{*}_{x}M$.

\begin{prop} \label{p:aequi} The set $\mathcal{A}$ is a non-empty open dense subset of $T^{*}M$.
\end{prop}
\begin{proof} Denote by $\mathcal{F}^{i}_{\lam}:=\mathcal{F}^{i}_{\gamma(0)}$ where $\g$ is the trajectory associated with initial covector $\lam$ and set $k_{i}(\lambda)=\dim \mathcal{F}^{i}_{\lam}$. By semicontinuity of the rank the integer valued and bounded function $\lambda\mapsto k_{i}(\lambda)$ is locally constant on an open dense set $\Omega_{i}$ of $T^{*}M$. Since the intersection of a finite number of open dense sets is open and dense, if follows that the set $\Omega=\cap_{i} \Omega_{i}$ where the growth vector is locally constant is open and dense in $T^{*}M$. To prove that it is non-empty fix an arbitrary point $x\in M$ and consider a $\lambda\in T^{*}_{x}M$ such that the corresponding trajectory is ample for all $t$ (the existence of such a trajectory is proved in \cite[Section 5.2]{curvature}). Since the functions $t\mapsto \dim \mathcal{F}^{i}_{\lam(t)}$ are lower semi-continuous and bounded with respect to $t$, repeating the previous argument we have that they are locally constant on an open dense set of $[0,T]$, then the curve is equiregular at these points.    
\end{proof}
We stress once more that for a fixed $x\in M$ one can have $\mathcal{A}_{x}=\emptyset$, as for instance in the non equiregular case. On the other hand, for every fixed $x$ the set $\lambda$ such that the corresponding trajectory is ample is open and dense and on each of these trajectories we can find equiregular points arbitrarily close to $x$. 
\subsection{Homogeneity properties}

For all $c>0$, let $H_c := H^{-1}(c/2)$ be the level set of the Hamiltonian function. In particular $H_1$ is the unit cotangent bundle: the set of initial covectors associated with unit-speed geodesics. Since the Hamiltonian function is fiber-wise quadratic, we have the following property for any $c>0$
\begin{equation}\label{eq:commutation}
e^{t \vec{H}}(c \lambda) = c e^{c t\vec{H}}(\lambda).
\end{equation}
Let $\delta_c : T^*M \to T^*M$ be the dilation along the fibers $\delta_c(\lambda) = c\lambda$ (if we write $\lam=(p,x)$ this means $\delta_c(p,x) = (cp,x)$). Indeed $\alpha \mapsto \delta_{e^\alpha}$ is a one-parameter group of diffeomorphisms. Its generator is the \emph{Euler vector field} $\mathfrak{e} \in \Gamma(\ver)$, and is characterized by $\delta_c  = e^{(\ln c)\mathfrak{e}}$.
We can rewrite~\eqref{eq:commutation} as the following commutation rule for the flows of $\vec{H}$ and $\mathfrak{e}$:
\begin{equation}
e^{t\vec{H}} \circ \delta_c = \delta_c \circ e^{ c t \vec{H}}.
\end{equation}
Observe that $\delta_c$ maps $H_1$ diffeomorphically on $H_{c}$. Let $\lambda \in H_1$ be associated with an ample, equiregular geodesic with Young diagram $\y$. Clearly also the geodesic associated with $\lambda^c:=c \lambda \in H_{c}$ is ample and equiregular, with the same Young diagram. This corresponds  to a reparametrization of the same curve: in fact $\lambda^c(t) =e^{t\vec{H}}(c\lam)= c(\lambda(ct))$, hence $\gamma^c(t) = \pi(\lambda^c(t))= \gamma(ct)$. The canonical frame associated with $\lambda^c(t)$ can be recovered by the one associated with $\lambda(t)$ as shown in the following proposition. Its proof can be found in \cite{BarRizJac}. 

%

\begin{prop}\label{p:framescaling}
Let $\lambda \in H_1$ and $\{E_{ai},F_{ai}\}_{ai \in \y}$ be the associated canonical frame along the extremal $\lambda(t)$. Let $c>0$ and define, for $ai \in \y$
\begin{equation}
E^c_{ai}(t):=\frac{1}{c^i}(d_{\lambda(ct)} P_{c}) E_{ai}(ct), \qquad F_{ai}^c(t):=c^{i-1}(d_{\lambda(ct)} \delta_c) F_{ai}(ct).
\end{equation}
The moving frame $\{E^c_{ai}(t),F^c_{ai}(t)\}_{ai \in \y} \in T_{\lambda^c(t)}(T^*M)$ is a canonical frame associated with the initial covector $\lambda^c =c\lambda\in H_{c}$, with matrix
\begin{equation}\label{eq:homog}
R^{\lambda^{c}}_{ai,bj}(t) = c^{i+j} R^\lambda_{ai,bj}(ct).
\end{equation}
\end{prop}

By  Proposition \ref{p:framescaling}, it follows the following homogeneity property of $g_{\lambda}$, and as a consequence of the function $\rho$.
\begin{lemma}\label{l:sgsg}
For every $\lambda\in \mathcal{A}_{x}$ and $c>0$ one has $c\lambda\in \mathcal{A}_{x}$. Moreover 
$$e^{g_{c\lambda}(t)}=c^{\calQ-n}e^{g_{\lambda}(ct)},$$
where $n=\dim M$ and $\calQ:=\sum_{i=1}^{m} id_{i}$. 
\end{lemma}
\begin{proof}
Let $X^c_{ai}(t)$ and $X_{ai}(ct)$ be the basis of $T_{\g^c(t)}M=T_{\g(ct)}M$ induced by the canonical frame. Then by Proposition \ref{p:framescaling} it holds the identity $X^c_{ai}(t)=c^{i-1}X_{ai}(ct).$ Therefore by the definition of $g_\lambda$ and $g_{c\lambda}$ we have
\begin{align}
e^{g_{c\lambda}(t)}&=|\mu_{\g(ct)}(X^c_{11}(t),\ldots,X^c_{kn_k}(t))|\\
&=\prod_{i=1}^m\prod_{j=1}^{d_i} c^{i-1} |\mu_{\g(ct)}(X_{11}(ct),\ldots,X_{kn_k}(ct))|= \prod_{i=1}^m c^{(i-1)d_{i}} e^{g_{\lambda}(ct)}\\
&=c^{\calQ-n} e^{g_{\lambda}(ct)}. \qedhere
\end{align}

\end{proof}
Lemma \ref{l:sgsg} gives $g_{c\lambda}(t)=g_{\lambda}(ct)+(\calQ-n)\log(c)$ and differentiating at $t=0$ we obtain
\begin{cor} For every $\lambda\in \mathcal{A}_{x}$ and $c>0$ one has
\begin{equation}\label{eq:rhoscale}
\rho(c\lambda)=c\rho(\lambda) .
\end{equation}
\end{cor}
\begin{remark} The function $\rho$ is homogeneous of degree one but, in general, it might not be smooth. Indeed using  formula \eqref{eq:MMM} and denoting by $M_{i}^{c}(t)$ the matrices associated with the reparametrized curve $\gamma^c$, one can show from the homogeneity properties of Proposition \ref{p:framescaling} that
\begin{equation}\label{eq:MMM2}
M_{i}^{c}(t)=c^{2i-2}M_{i}(ct),\qquad \dot M_{i}^{c}(t)=c^{2i-1}\dot M_{i}(ct)
\end{equation}
from which it follows that $\rho$ is a rational function in $\lambda$ with the degree of the denominator which is at most $\lam^{2m-2}$. Notice that using \eqref{eq:MMM2} at $t=0$ one can obtain another proof of \eqref{eq:rhoscale} by
$$\rho(c\lam)=\frac12 \sum_{i=1}^{m}\tr\left(M^{c}_i(0)^{-1}\dot{M}^{c}_i(0)\right)=\frac{1}{2}\sum_{i=1}^m \tr\left(c M_i(0)^{-1}\dot{M}_i(0)\right) =c\rho(\lam).$$
\end{remark}

\subsection{Contact manifolds}\label{sec:contact} 
In this section we focus on the special case of a contact sub-Riemannian manifold. Recall that a sub-Riemannian manifold $(M,\calD,g)$ of odd dimension $2n+1$ is \emph{contact} if there exists a non degenerate 1-form $\omega\in\Lambda^1(M)$, such that $\calD_x=\ker \omega_x$ for every $x\in M$ and $d\omega|_{\calD}$ is non degenerate. In this case $\calD$ is called \emph{contact distribution}. 

\begin{remark}
Given a sub-Riemannian contact manifold, the contact form $\omega$ is not unique. Indeed if $\omega$ is a contact form then also $\alpha\omega$ is a contact form for any non-vanishing smooth function $\alpha$. Once a contact form $\omega$ is fixed we can associate the \emph{Reeb vector field}, $X_0$, which is the unique vector field such that $\omega(X_0)=1$ and $d\omega(X_0,\cdot)=0$. Since the Reeb vector field $X_0$ is transversal to $\calD$, we normalize $\omega$ so that $\|X_0\|_{\calD^2/\calD}=1$.
\end{remark}

The contact form $\omega$ induces a fiber-wise linear map $J:\calD\to\calD$, defined by
$$\langle JX,Y\rangle=d\omega(X,Y)\qquad \forall X,Y\in\calD.$$
Observe that the restriction $J_x:=J|_{\distr_{x}}$ is a linear skew-symmetric operator on $(\calD_x,g_{x})$. 

Let $X_1,\ldots,X_{2n}$ be a local orthonormal frame of $\calD$, then $X_1,\ldots,X_{2n},X_0$ is a local  frame adapted to the flag of the distribution. Let $\nu^1,\ldots,\nu^{2n},\nu^0$ be the associated dual frame. The \emph{Popp volume} $\mu$ on $M$ (see \cite{nostropopp} for more details) is the volume
\begin{equation}
\mu=\nu^1\wedge\ldots\wedge\nu^{2n}\wedge\nu^0.
\label{eq:contact popp}
\end{equation}

On contact sub-Riemannian manifolds, every non constant geodesic $\g(t)=\pi (e^{t\vec{H}}(\lambda))$ is ample and equiregular with the same growth vector $(2n,2n+1)$. Moreover, it is possible to compute explicitly the value of the associated smooth function $g_\lambda(t)$ and the constant $C_\lam$ of Theorem \ref{th:main} (cf.\ Remark \ref{remark:rema}).

We compute now the value of the function $g_\lambda(t)$ with respect to the Popp's volume and a given geodesic. 
\begin{prop}Let $\gamma(t)=\pi (e^{t\vec{H}}(\lambda))$ be a geodesic on a contact manifold. Then
\begin{equation}
\rho(\lambda)=\frac{d}{dt}\bigg|_{t=0}\log \|J\dot\g(t)\|,
\label{eq:g log}
\end{equation}
In particular if $J^{2}=-1$ then $\rho=0$.
\end{prop}
\begin{proof}
Recall that, by definition of $g_\lambda$ (cf.\ \eqref{eq:g}), one has
\begin{equation}
g_\lambda(t)=\log |\mu(P_t)|
\label{eq:g log}
\end{equation}
where $P_t$ is the parallelotope whose edges are given by the projections $X_{ai}(t)$  of the fields $F_{ai}(t)$ of a canonical basis along $\lambda(t)$ on $T_{\g(t)}M$. 

Let $\tanf$ be an horizontal extension of the tangent vector field $\dot{\g}(t)$ and consider the map $\calL_\tanf:\calD_{\g(t)}\to \calD^2_{\g(t)}/\calD_{\g(t)}$. Since the manifold is contact, this map is surjective. and its kernel is a subspace of $\calD_{\g(t)}$ of dimension $2n-1$. Let $X_1,\ldots,X_{2n}$ be an orthonormal basis of $\calD_{\g(t)}$ such that $X_1\in (\ker \calL_\tanf)^\perp$ and  $X_2,\ldots,X_{2n}\in \ker \calL_\tanf$. Then  there exists an orthogonal map that transforms the first $2n$ vectors projections of the canonical basis, in this basis. 

Notice that the definition \eqref{eq:g log} of $g_\lambda(t)$ does not change if we replace the first $2n$ edges of the parallelotope by $X_1,\ldots,X_{2n}$. Moreover, by Lemma \ref{lemma:Xai}, the last projected vector $X_{ai}=X_{1,2}$ can be written as 
$$X_{1,2}(t)=-\calL_\tanf X_1(t) \mod \calD.$$ 
Notice that $X_1$ is not in the kernel of $\calL_\tanf$, thus this basis is also adapted to the Young diagram of $\g$. Thanks to \eqref{eq:contact popp}, the Popp volume of the parallelotope is equal to the length of the component of $\calL_\tanf X_1(t)$ with respect to $X_0$, namely
$$|\mu(P_t)|=|\langle [\tanf,X_1],X_0\rangle_{\g(t)}|.$$
This quantity can be written equivalently in terms of the map $J$. Indeed
\begin{equation}
\begin{split}
|\mu(P_t)|&=|\langle [\tanf,X_1],X_0\rangle_{\g(t)}|=|\omega_{\g(t)}([\tanf,X_1])|=|d\omega_{\g(t)}(\tanf,X_1)|\\
&=|\langle J_{\g(t)}\tanf,X_1\rangle_{\g(t)}|.
\end{split}
\end{equation}
Since $\langle J\tanf,Y\rangle=-\omega(\calL_\tanf Y)$ for every horizontal field $Y$, then $\ker \calL_\tanf=J\tanf^\perp$. This implies that $J\tanf$ is a multiple of $X_1$, i.e., $J\tanf=\|J\tanf\| X_1$. Then we simplify the formula for $|\mu(P_t)|$ as
$$|\mu(P_t)|=|\langle J_{\g(t)}\tanf,X_1\rangle_{\g(t)}|=\|J\tanf_{\g(t)}\|.$$
Notice that if $J^{2}=-1$, then $J$ is an isometry, hence $\|J\tanf_{\g(t)}\|=\|\tanf_{\g(t)}\|=1$.
\end{proof}
\begin{remark}\label{rem:popprho3d} If $\dim M=3$, then $\rho(\lambda(t))=0$  for every $t$. Indeed $\ker \calL_\tanf$ has dimension 1 and $\tanf=\|\tanf\| X_2$. Moreover, if we denote by $c_{ij}^k$ the structure constants such that $[X_i,X_j]=\sum_{k=0}^2 c_{ij}^k X_k$, then the normalization of $\omega$ implies $c_{12}^0=-1$ and
\begin{equation}
\begin{split}
|\mu(P_t)|&=|\langle [\tanf,X_1],X_0\rangle|=\|\tanf_{\g(t)}\|\, |\langle[X_2,X_1],X_0\rangle|\\
&=\|\tanf_{\g(t)}\| =1.
\end{split}
\end{equation}\end{remark}

\begin{remark} Notice that even in the contact case, not every structure is unimodular (in the sense of Definition \ref{d:unimod}). When $J^{2}=-1$ then the structure is unimodular, choosing $\mu$ as the Popp volume. See also \cite{ABR-Contact} for the computation of the curvature in the contact case.
\end{remark}

\section{Proof of the main result} \label{s:proof}
In this section we prove the following proposition, that is Theorem \ref{th:main} written along the canonical frame. 
\begin{prop}\label{prop:main}
Let $\g(t)=\pi(e^{t\vec{H}}(\lambda))$ be an ample equiregular geodesic and let $\omega_{\g(t)}$ be the $n$-form defined in \eqref{eq:volumeomega}. Given a volume $\mu$ on $M$, define implicitly the smooth function $g_{\lam}:[0,T]\to \bbR$ by $\mu_{\g(t)}=e^{g_{\lam}(t)}\omega_{\g(t)}$. Then we have the following Taylor expansion
\begin{equation}
\left.\left\langle \left(\pi\circ e^{t\vec{H}}\right)^*\mu, E(0)\right\rangle\right|_{\lambda}=\;C_{\lambda}t^{\mc{N}(\lambda)}e^{g_{\lam}(t)}\left(1-t^2\frac{\mathrm{tr} \mc{R}_\lambda}{6}+o(t^2) \right)
\label{eq:volumevariation}
\end{equation}
where $E$ is the $n$-dimensional row vector introduced in Remark \ref{rmk:notation} and $C_{\lambda}$ depends only on the structure of the Young diagram. In particular we have the identity
$$\rho(\lam)\mu_{\lam}^{*}= \frac{d}{dt}\bigg|_{t=0} \log \left(t^{-\mathcal{N}(\lambda)}\left.(\pi\circ e^{t\vec H})^{*}\mu\right|_{\V_{\lambda}}\right)=\dot g_{\lam}(0)\mu_{\lam}^{*}.$$
\end{prop}
\begin{remark}\label{remark:rema}
As it follows from the proof, the constant $C_{\lambda}$ is explicitly computed  by
\begin{equation}
C_\lambda=\prod_{a=1}^k \frac{\prod_{j=0}^{n_a-1}j!}{\prod_{j=n_a}^{2n_a-1}j!}>0.
\label{eq:c1}
\end{equation} 
In the contact case, since the Young diagram is equal for all $\lambda$, with $2n$ rows of length 1, and one row of length $2$, the leading constant is $C_{\lambda}=\frac{1}{12}$.
\end{remark}

\begin{proof}
The left hand side of the equation \eqref{eq:volumevariation} can be rewritten as
$$\left.\left\langle \left(\pi \circ e^{t\vec{H}}\right)^*\mu, E(0)\right\rangle\right|_{\lambda}=\left.\left\langle e^{g(t)}\omega,  \left(\pi \circ e^{t\vec{H}}\right)_*E(0)\right\rangle\right|_{\g(t)}.$$
 
For every $ai\in D$, the field $e^{t\vec{H}}_*E_{ai}(0)$ is a Jacobi field, so in coordinates with respect to the canonical frame we can write 
$$e^{t\vec{H}}_*E(0)=E(t)M(t)+F(t)N(t)$$ 
for $n\times n$ matrices $M$ and $N$, that satisfy the Jacobi equations \eqref{eq:Jacobicoord}. More explicitly we have the system
\begin{equation}
\left\{\begin{array}{ll}
\dot{N}_{ai,bj}=N_{ai-1,bj} & \mbox{ if } i\neq 1\\
\dot{N}_{a1,bj}=M_{a1,bj}\\
\dot{M}_{ai,bj}=-R(t)_{ai,ch}N_{ch,bj}-M_{ai+1,bj} & \mbox{ if } i\neq n_a\\
\dot{M}_{an_a,bj}=-R(t)_{an_a,ch}N_{ch,bj}.
\end{array}
\right.
\label{eq:derivateMN}
\end{equation}
Moreover $M(0)=\mathrm{Id}$ and $N(0)=0$. It follows that 
$$\left.\left\langle \left(\pi \circ e^{t\vec{H}}\right)^*\mu, E(0)\right\rangle\right|_{\lambda}=e^{g(t)}\det N(t).$$
In what follows we compute the Taylor expansion of the matrix $N(t)$ at $0$. 

Let us first discuss the proof in the case of a Young diagram made of a single row. In this case, for simplicity, we will omit the index $a$ in the notation for $N$ and $M$. Fix integers $1\leq i,j\leq n$. The coefficients $N_{ij}$ can be computed by integrating  $M_{1j}$. So let us find the asymptotic expansion of $M_{1j}$. Notice that
\begin{equation}
\begin{split}
&M_{1j}(0)=\delta_{1j}\\
&\dot{M}_{1j}=-R_{1h}N_{hj}-M_{2j}(1-\delta_{1n})\\
&\ddot M_{1j}=-\dot{R}_{1h}N_{hj}-\sum_{h\neq 1}R_{1h}N_{h-1j}-R_{11}M_{1j}+(1-\delta_{1n})\left(R_{2h}N_{hj}+M_{3j}(1-\delta_{2n})\right)
\end{split}
\end{equation}
In these equations the only non-vanishing component at $t=0$ is $M_{jj}(0)=1$, that can be obtained only by differentiating terms $M_{ij}$ with $i<j$. Thus, in the expansion of $M_{1j}(t)$, the element $M_{jj}$ appears first at $(j-1)$-th derivative. Next, it appears, multiplied by $R_{11}(0)$, at $(j+1)$-th derivative.
We can conclude that the asymptotics of $M_{1j}$ at $t=0$  is
$$M_{1j}(t)=(-1)^{j-1}\frac{t^{j-1}}{(j-1)!}-(-1)^{j-1}R_{11}(0)\frac{t^{j+1}}{(j+1)!}+o(t^{j+1}).$$
Since $M_{1j}$ is the $i$-th derivative of $N_{ij}$ and $N(0)=0$, we have also the expansion for $N$:
$$N_{ij}(t)=(-1)^{j-1}\frac{t^{i+j-1}}{(i+j-1)!}-(-1)^{j-1}R_{11}(0)\frac{t^{i+j+1}}{(i+j+1)!}+o(t^{i+j+1}).$$

Let us now consider a general distribution of dimension $k> 1$. Now we have to study the whole system in \eqref{eq:derivateMN}. Fix indeces $ai,bj\in D$. Again, to find  $N_{ai,bj}$ it's enough to determine the expansion of $M_{a1,bj}$, by integration. To compute the latter, notice that 
\begin{equation}
\begin{split}
&M_{a1,bj}(0)=\delta_{ab}\delta_{1j}\\
&\dot{M}_{a1,bj}=-R_{a1,ch}N_{ch,bj}-M_{a2,bj}(1-\delta_{1n_a})\\
&M_{a1,bj}^{(2)}=-\dot{R}_{a1,ch}N_{ch,bj}-\sum_{h\neq 1}R_{a1,ch}N_{ch-1,bj}-R_{a1,c1}M_{c1,bj}\\
&\qquad \qquad+(1-\delta_{1n_a})\left(R_{a2,ch}N_{ch,bj}+M_{a3,bj}(1-\delta_{2n_a})\right)
\end{split}
\end{equation}
When $a= b$,   the argument is similar to the one discussed above when $k=1$ (in this case every derivative  generates also terms like $M_{ch,aj}$, but these terms, when $c\neq a$, need higher order derivatives to generate non-vanishing terms). One obtains:
\begin{equation*}
\begin{split}
M_{a1,aj}(t)&=(-1)^{j-1}\frac{t^{j-1}}{(j-1)!}-(-1)^{j-1}R_{a1,a1}\frac{t^{j+1}}{(j+1)!}+o(t^{j+1}),\\
N_{ai,aj}(t)&=(-1)^{j-1}\frac{t^{i+j-1}}{(i+j-1)!}-(-1)^{j-1}R_{a1,a1}\frac{t^{i+j+1}}{(i+j+1)!}+o(t^{i+j+1}).
\end{split}
\end{equation*}
On the other hand, if $a\neq b$, then the first term different from zero of $M_{a1,bj}$ appears at  $j+1$-th derivative, multiplied by $R_{a1,b1}$, 
Therefore the Taylor expansions of $M_{ai,bj}$ and of a generic element of the matrix $N$ can be derived as
\begin{equation*}
\begin{split}
M_{a1,bj}(t)&=\delta_{ab}(-1)^{j-1}\frac{t^{j-1}}{(j-1)!}-(-1)^{j-1}R_{a1,b1}(0)\frac{t^{j+1}}{(j+1)!}+o(t^{j+1}),\\
N_{ai,bj}(t)&=\widetilde{N}_{ai,bj}t^{i+j-1}-G_{ai,bj}t^{i+j+1}+o(t^{i+j+1}).
\end{split}
\end{equation*}
where the constant matrices $\widetilde{N}$ and $G$ are defined by
\begin{equation*}
\widetilde{N}_{ai,bj}:=(-1)^{j-1}\frac{\delta_{ab}}{(i+j-1)!},\qquad  G_{ai,bj}:=(-1)^{j-1}\frac{R_{a1,b1}(0)}{(i+j+1)!}.
\end{equation*}

To find the asymptotics of the left hand side of \eqref{eq:volumevariation}, we need only to determine the asymptotic of $\det N(t)$.
Let $I_{\sqrt{t}}$ be a $n$-dimensional diagonal matrix, whose $jj$-th element is equal to $\sqrt{t}^{2i-1}$, for $k_{i-1}<j\leq k_i$. Then the Taylor expansion of $N$ can be written as $N(t)=I_{\sqrt{t}}\left(\widetilde{N}-t^2 G+O(t^3)\right)I_{\sqrt{t}}$ and its determinant is
$$\det N(t)=\det \widetilde{N}\; t^{\calN}\left(1-\tr\left(\widetilde{N}^{-1} G\right)t^2+o(t^2)\right),$$
where $\calN=\calN(\lambda)$ is the geodesic dimension given in Definition \ref{def:geodesic dimension}. 
The main coefficient is computed in the following lemma, whose proof is contained  in Appendix \ref{s:appA}.
\begin{lemma}\label{th:conti}
The determinant of $\widetilde{N}$ is given by
\begin{equation}
C_{\lambda}=\det \widetilde{N}=\prod_{a=1}^k \frac{\prod_{j=0}^{n_a-1}j!}{\prod_{j=n_a}^{2n_a-1}j!}.
\label{eq:c0}
\end{equation}
\end{lemma}
Since the matrix $\widetilde{N}$ is block-wise diagonal, to find the trace of $\widetilde{N}^{-1} G$ we just need the elements of $G$ with $a=b$. Thanks to \eqref{eq:ROmega}, that relates the curvature operator $\calR_\lambda$ with the elements of the matrix $R_{a1,b1}$, we have
$$\calR_{aa}=3 \frac{n_a}{4 n_a^2-1}R_{a1,a1}(0).$$
Moreover we have the following identity
\begin{equation}
\begin{split}
\tr\left(\widetilde{N}^{-1} G\right)&=\sum_{a=1}^k \left(\sum_{i,j=1}^{n_a}[\widetilde{N}^{-1}]_{ai,aj}\frac{(-1)^{j-1}}{(i+j+1)!}\right)R_{a1,a1}(0).
\end{split}
\label{eq:traceN}
\end{equation}
where we denoted by $[\widetilde{N}^{-1}]_{ai,bj}$ the $(ai,bj)$ entry of the matrix $\widetilde{N}^{-1}$. 
The proof of the statement is then reduced to the following lemma, whose proof is  postponed in Appendix \ref{s:appB}.
\begin{lemma}\label{th:conti1}
Let $\widehat{N}$ and $\widehat{G}$ be $n\times n$ matrices, whose elements are $\widehat{N}_{ij}=\frac{(-1)^{j-1}}{(i+j-1)!}$ and $\widehat{G}_{ij}=\frac{(-1)^{j-1}}{(i+j+1)!}$. Then 
\begin{equation}
\tr\left(\widehat{N}^{-1} \widehat{G}\right)=\frac{1}{2}\frac{n}{4 n^2-1}.
\label{eq:traceN2}
\end{equation}
\end{lemma}
\end{proof}


\appendix
\section{Proof of Lemma \ref{th:conti}}\label{s:appA}

We compute the value of the leading constant $C_{\lambda}:=\det \widetilde{N}$. Recall that $\widetilde{N}$ is a block matrix, whose only non-vanishing blocks are the diagonal ones. Moreover, every $aa$-block of the diagonal is the matrix $\widehat{N}$ of dimension $n_a$. Thus, to find the determinant of $\widetilde{N}$, it is sufficient to evaluate the determinant of the generic matrix $\widehat{N}$ of dimension $n$ defined by
$$\widehat{N}_{ij}=\frac{(-1)^{j-1}}{(i+j-1)!}$$

The matrix $\widehat{N}$ has already been studied in \cite{curvature}, Section 7.3 and Appendix G, and its inverse can be expressed as a product of two matrices $\widehat{N}^{-1}_{ij}=\left(\widehat{S}^{-1}\widehat{A}^{-1}\right)_{ij}$, where 
\begin{gather}
\widehat{A}^{-1}_{ij}:=\frac{(-1)^{i-j}}{(i-j)!} \qquad i\geq j,\\
\widehat{S}^{-1}_{ij}:=\frac{1}{i+j-1}\binom{n+i-1}{i-1}\binom{n+j-1}{j-1}\frac{(n!)^2}{(n-i)!(n-j)!}.
\end{gather}
Therefore the inverse of $\widehat{N}$ is
\begin{equation}\label{eq:4a}
\widehat{N}^{-1}_{ij}=\sum_{k=j}^n\frac{(-1)^{k-j}}{(i+k-1)(k-j)!}{n+i-1\choose i-1}{n+k-1\choose k-1}\frac{(n!)^2}{(n-i)!(n-k)!}.
\end{equation}
By Cramer's rule one obtains
\begin{equation}\label{eq:ultima}
\widehat{N}^{-1}_{ij}=(-1)^{i+j}\frac{\det \widehat{N}^0_{ji}}{\det \widehat{N}},
\end{equation}
where $\widehat{N}^0_{ji}$ is the matrix of dimension $n-1$ obtained from $\widehat{N}$ by removing the $j$-th row and the $i$-th column. Applying \eqref{eq:ultima} for $i=j=n$ we reduce the computation of the determinant of the matrix $\widehat{N}$ of dimension $n$ as the product of the $(n,n)$-entry of the matrix $\widehat{N}^{-1}$ and the the determinant of the matrix $\widehat{N}$ of dimension $n-1$, namely we get the recursive formula:
$$\det \widehat{N}_{(n)}=\frac{\det \widehat{N}_{(n-1)}}{\widehat{N}^{-1}_{nn}}=\det \widehat{N}_{(n-1)} \frac{(n-1)!^2}{(2n-2)!(2n-1)!},$$
where the last equality follows from equation \eqref{eq:4a}. Using that, for $n=1$, the determinant is equal to $1$,  we obtain the general formula
\begin{equation}
\det \widehat{N}_{(n)}=\frac{\prod_{j=0}^{n-1}j!}{\prod_{j=n}^{2n-1}j!}.
\label{eq:detB}
\end{equation}
The value of the constant $C_\lambda$ is then obtained as the product
\begin{equation}
C_\lambda=\prod_{a=1}^k \det \widehat{N}_{(n_a)}=\prod_{a=1}^k \frac{\prod_{j=0}^{n_a-1}j!}{\prod_{j=n_a}^{2n_a-1}j!}.
\end{equation}This concludes the proof of Lemma \ref{th:conti}.

\section{Proof of Lemma \ref{th:conti1}}\label{s:appB}
The goal of this appendix is to prove Lemma \ref{th:conti1}, i.e., the identity
\begin{gather}\label{eq:222}
\tr\left(\widehat{N}^{-1} \widehat{G}\right)=\sum_{i,j=1}^n \widehat{N}^{-1}_{ij}\widehat{G}_{ji}=\frac{1}{2}\frac{n}{4 n^2-1}.
\end{gather}
where we recall that the matrices  $\widehat{N}$ and $\widehat{G}$ are defined by
\begin{equation}
\widehat{N}_{ij}=\frac{(-1)^{j-1}}{(i+j-1)!},\qquad \widehat{G}_{ij}=\frac{(-1)^{j-1}}{(i+j+1)!}.
\end{equation}
Using the formula \eqref{eq:4a} for the expression of $\widehat{N}_{ij}^{-1}$, we are reduced to prove the following combinatorial identity
\begin{equation} 
\sum_{i=1}^{n}\sum_{j=1}^n\sum_{k=j}^n\frac{(-1)^{k-j}}{(i+k-1)(k-j)!}{n+i-1\choose i-1}{n+k-1\choose k-1}\frac{(n!)^2}{(n-i)!(n-k)!}\frac{(-1)^{i-1}}{(i+j+1)!}=\frac{1}{2}\frac{n}{4 n^2-1}.
\label{eq:traceN3}
\end{equation}
It is immediate to check that for $n=1$ the identity \eqref{eq:traceN3} is true. Then in what follows we assume that $n\geq 2$. 

Notice that for $i=1,\ldots,n-2,$ we have $\widehat{G}_{ji}=\widehat{N}_{j(i+2)}$ therefore this sum reduces to the sum of the components with $i=n-1$ and $i=n$:
\begin{align*}
\sum_{i,j=1}^n \widehat{N}^{-1}_{ij}\widehat{G}_{ji}&=\sum_{i=1}^{n-2}\sum_{j=1}^n \widehat{N}^{-1}_{ij}\widehat{G}_{ji}+\sum_{i=n-1}^{n}\sum_{j=1}^n \widehat{N}^{-1}_{ij}\widehat{G}_{ji}\\
&=\sum_{i=1}^{n-2}\sum_{j=1}^n \widehat{N}^{-1}_{ij}\widehat{N}_{j(i+2)}+\sum_{i=n-1}^{n}\sum_{j=1}^n \widehat{N}^{-1}_{ij}\widehat{G}_{ji}\\
&=\sum_{i=1}^{n-2} \delta_{i(i+2)}+\sum_{i=n-1}^{n}\sum_{j=1}^n \widehat{N}^{-1}_{ij}\widehat{G}_{ji}\\
&=\sum_{i=n-1}^{n}\sum_{j=1}^n \widehat{N}^{-1}_{ij}\widehat{G}_{ji},
\end{align*}
where $\delta_{ij}$ is the Kronecker symbol.
In particular, our initial claim \eqref{eq:traceN2}  follows by summing the next two combinatiorial identities, that are valid for $n\geq2 $:  
\begin{equation}
\sum_{j=1}^n\sum_{k=j}^n\frac{(-1)^{k-j+n}}{(n+k-2)(k-j)!}{2n-2\choose n-2}{n+k-1\choose k-1}\frac{(n!)^2}{(n-k)!(n+j)!}=-\frac{n-1}{4(2n-1)},
\label{eq:6}
\end{equation}
\begin{equation}
\sum_{j=1}^n\sum_{k=j}^n\frac{(-1)^{k-j+n+1}}{(n+k-1)(k-j)!}{2n-1\choose n-1}{n+k-1\choose k-1}\frac{(n!)^2}{(n-k)!(n+j+1)!}=\frac{n+1}{4(2n+1)}.
\label{eq:7}
\end{equation}
Before proving \eqref{eq:6} and \eqref{eq:7}, let us first simplify them. Using the following binomial identity, for $n\geq 2$ and $1\leq j\leq k\leq n$ 
$$\frac{1}{(k-j)!}{2n-2\choose n-2}\frac{(n!)^2}{(n-k)!(n+j)!}\frac{2(2n-1)}{n-1}={2n \choose n+k}{n+k\choose n+j},$$
one obtains that \eqref{eq:6} is equivalent to 
\begin{equation} \label{eq:cisiamo1}
\begin{split}
&\sum_{j=1}^n\sum_{k=j}^n\frac{(-1)^{n+k+j}}{n+k-2}{2n \choose n+k}{n+k\choose n+j}{n+k-1\choose k-1}=-\frac{1}{2}.
\end{split}
\end{equation}
\begin{lemma} \label{l:b0} For $k,n\geq 1$, one has the following combinatorial identity
\begin{equation}
\begin{split}
\sum_{j=1}^k(-1)^j{n+k\choose n+j}&=-{n+k-1 \choose k-1}.
\end{split}
\end{equation}
\end{lemma}
\begin{proof}[Proof of Lemma \ref{l:b0}] We have the following chain of identities
\begin{equation}
\begin{split}
\sum_{j=1}^k(-1)^j{n+k\choose n+j}&=(-1)^n\sum_{j=n+1}^{n+k}(-1)^j{n+k\choose j}\\
&=-(-1)^n\sum_{j=0}^{n}(-1)^j{n+k\choose j}
=-{n+k-1 \choose n}=-{n+k-1 \choose k-1},
\end{split}
\end{equation}
where in the first one we used a change of variable $j\to j+n$ in the sum, while in the second one we used the general identity $0=(-1+1)^N=\sum_{j=0}^N(-1)^j{N\choose j}$. The last equality follows from the identity 
$$\sum_{j=0}^n (-1)^j{n+k\choose j}=(-1)^n{n+k-1\choose n},$$
that can be easily proved for every fixed $k$, by induction on $n\geq 1$.
\end{proof}
Applying Lemma \ref{l:b0} to the left hand side of \eqref{eq:cisiamo1}, we have
\begin{equation*}
\begin{split}
&\sum_{j=1}^n\sum_{k=j}^n\frac{(-1)^{n+k+j}}{n+k-2}{2n \choose n+k}{n+k\choose n+j}{n+k-1\choose k-1}=\\
&=\sum_{k=1}^n\left[\sum_{j=1}^k(-1)^j{n+k\choose n+j}\right]\frac{(-1)^{n+k}}{n+k-2}{2n \choose n+k}{n+k-1\choose k-1}\\
&=-\sum_{k=1}^n\frac{(-1)^{n+k}}{n+k-2}{2n \choose n-k}{n+k-1\choose k-1}^2.
\end{split}
\end{equation*}
Thus we have proved that  \eqref{eq:cisiamo1}, which is equivalent to \eqref{eq:6}, is also equivalent to
\begin{equation}
\sum_{k=1}^n\frac{(-1)^{n+k}}{n+k-2}{2n \choose n-k}{n+k-1\choose k-1}^2=\frac{1}{2}.
\label{eq:81}
\end{equation}
Performing analogous  transformations, one proves that \eqref{eq:7} is equivalent to
\begin{equation}
\sum_{k=1}^n\frac{(-1)^{n+k}}{(n+k)(n+k-1)}{2n+1 \choose n-k}{n+k\choose k-1}^2=\frac{1}{2}.
\label{eq:91}
\end{equation}
The proof is then completed thanks to the next lemma.
\begin{lemma} \label{l:b1} For $n\geq 2$, the following combinatorial identities hold
\begin{gather}\label{eq:8}
\sum_{k=1}^n\frac{(-1)^{n+k}}{n+k-2}{2n \choose n-k}{n+k-1\choose k-1}^2=\frac{1}{2},\\
\sum_{k=1}^n\frac{(-1)^{n+k}}{(n+k)(n+k-1)}{2n+1 \choose n-k}{n+k\choose k-1}^2=\frac{1}{2}.
\label{eq:9}
\end{gather}
\end{lemma}
\begin{proof}[Proof of Lemma \ref{l:b1}]
Let us first prove  \eqref{eq:8}. Denote by 
$$\beta_{k,n}:=\frac{(-1)^{n+k}}{n+k-2}{2n \choose n-k}{n+k-1\choose k-1}^2,$$ the coefficient appearing in the sum  \eqref{eq:8}, we want to prove $\sum_{k=1}^{n}\beta_{k,n}=\frac12$.
To this aim, we apply Lemma \ref{l:hilbert} to two different matrices.

Let us first consider the Hilbert matrix $H_1:=\left[\frac{1}{i+j-1}\right]_{ij}$. It is a matrix of type \eqref{eq:10}, with $a_i:=i$ and $b_j:=-j+1$. We compute the coefficients of the $(n-1)$-th row of $H_1^{-1}$:
\begin{equation}
\begin{split}
(H_1^{-1})_{n-1,j}&=\frac{1}{-n+2-j}\frac{\prod_k(-n+2-k)(j+k-1)}{\prod_{k\neq j}(j-k)\prod_{l\neq n-1}(-n+2+l-1)}\\
&=\frac{n(n-1)^2}{2(2n-1)}\;(n+j)\;\frac{(-1)^{n+j+1}}{n+j-2}{2n \choose n-j}{n+j-1\choose j-1}^2\\
&=-\frac{n(n-1)^2}{2(2n-1)}\;(n+j)\;\beta_{j,n}
\end{split}
\label{eq:13}
\end{equation}
Then consider $H_2:=\left[\frac{1}{a_i-b_j}\right]$, with $a_i=i$ for $i<n$ and $a_n=-n$, while $b_j=-j+1$. We compute the coefficients of the $(n-1)$-th row of $H_2^{-1}$. For $j<n$ we have
\begin{equation}
\begin{split}
(H_2^{-1})_{n-1,j}&=\frac{1}{-n+2-j}\frac{\prod_{k\neq n}(-n+2-k)(-n+2+n)\prod_k(j+k-1)}{\prod_{k\neq j, k\neq n}(j-k) \;(n+j)\prod_{l\neq n-1}(-n+2+l-1)}\\
&=\frac{n(n-1)}{2(2n-1)}\;(n-j)\;\frac{(-1)^{n+j+1}}{n+j-2}{2n \choose n-j}{n+j-1\choose j-1}^2\\
&=-\frac{n(n-1)}{2(2n-1)}\;(n-j)\;\beta_{j,n},
\end{split}
\label{eq:14}
\end{equation}
while for $j=n$ we get
\begin{equation}
\begin{split}
(H_2^{-1})_{n-1,n}=&\frac{1}{-n+2+n}\frac{\prod_{k\neq n}(-n+2-k)(-n+2+n)\prod_k(-n+k-1)}{\prod_{k\neq n}(-n-k) \;(n+j)\prod_{l\neq n-1}(-n+2+l-1)}=\frac{n^2(n-1)}{2(2n-1)}.
\end{split}
\label{eq:15}
\end{equation}
Setting 
$$\alpha_1:=\sum_{j=1}^n(H^{-1}_1)_{n-1,j},\qquad \alpha_2:=\sum_{j=1}^n(H^{-1}_2)_{n-1,j}$$
and summing over $j$ the identities  \eqref{eq:13} and \eqref{eq:14} one gets
$$\sum_{j=1}^n\beta_{j,n}=-\frac{1}{2n}\left[\frac{2(2n-1)}{n(n-1)^2}\alpha_1+\frac{2(2n-1)}{n(n-1)}\left(\alpha_2-\frac{n^2(n-1)}{2(2n-1)}\right)\right].$$
Now the proof of equation \eqref{eq:8} is completed once we use formula \eqref{eq:12} to find the values 
$$\alpha_1=-\frac{(2n-2)!}{(n-2)!^2},\qquad \alpha_2=\frac{2(2n-3)!}{(n-2)!^2}.$$
Equation \eqref{eq:9} can be proved along the same lines. More precisely, define
$$\gamma_{k,n}:=\frac{(-1)^{n+k}}{(n+k)(n+k-1)}{2n+1 \choose n-k}{n+k\choose k-1}^2$$ 
as the coefficients appearing in the sum \eqref{eq:9}, and consider the Hilbert matrix $H_1$, and the matrix $H_3$ obtained by $a_j=j$ if $j<n$ and $a_n=-n-1$, and $b_j=-j+1$. Then by Lemma \ref{l:hilbert}
$$(H_1^{-1})_{n,j}=(n+j+1)\,\frac{n(n+1)^2}{2(2n+1)}\gamma_{j,n}.$$
Moreover, for $j<n$
$$(H_3^{-1})_{n,j}=(n-j)\,\frac{n(n+1)^2}{(2n+1)(2n-1)}\gamma_{j,n},$$
while for $j=n$ $$(H_3^{-1})_{n,n}=-\frac{n(n+1)^2}{2(2n-1)}.$$
One can also compute 
$$\eta_1:=\sum_{j=1}^n(H^{-1}_1)_{n,j}=\frac{(2n-1)!}{(n-1)!^2},\qquad \eta_3:=\sum_{j=1}^n(H^{-1}_3)_{n,j}=-\frac{2(2n-2)!}{(n-1)!^2}.$$ 
Then the sum in \eqref{eq:9} is given by
\begin{equation}
\sum_{j=1}^n \gamma_{j,n}=\frac{1}{2n+1}\left[\frac{2(2n+1)}{n(n+1)^2}\eta_1+\frac{(2n+1)(2n-1)}{n(n+1)^2}\left(\eta_3+\frac{n(n+1)^2}{2(2n-1)}\right)\right]=\frac{1}{2}.
\end{equation}
that completes the proof of the lemma.
\end{proof}
The following lemma concerns the inverse of the generalized Hilbert matrix.
\begin{lemma}[see \cite{art:Schechter}] \label{l:hilbert}
Let $a_1,\ldots,a_n,b_1,\ldots,b_n$ be $2n$ distinct reals and define the
$n\times n$ matrix 
\begin{equation}
H_{ij}=\frac{1}{a_i-b_j}.
\label{eq:10}
\end{equation}
Then we have
\begin{equation}
(H^{-1})_{ij}=\frac{1}{b_i-a_j}\frac{\displaystyle \prod_{k=1}^n(b_i-a_k)(a_j-b_k)}{\displaystyle \prod_{k\neq j}(a_j-a_k)\prod_{l\neq i}(b_i-b_l)}, \qquad 
\sum_{j=1}^n(H^{-1})_{ij}=-\frac{\displaystyle \prod_{k=1}^n(b_i-a_k)}{\displaystyle \prod_{k\neq i}(b_i-b_k)}.
\label{eq:12}
\end{equation}
\end{lemma}

\addtocontents{toc}{\protect\setcounter{tocdepth}{0}}
\section*{Acknowledgments} 
This research has been supported by the European Research Council, ERC StG 2009 ``GeCoMethods'', contract number 239748 and by the ANR project SRGI ``Sub-Riemannian Geometry and Interactions'', contract number ANR-15-CE40-0018.

%

	\bibliographystyle{abbrv}
	\bibliography{bibliografia}		
\end{document}